\documentclass[journal,twoside,web]{ieeecolor}
\usepackage{generic}
\usepackage{cite}
\usepackage{amsmath,amssymb,amsfonts}
\usepackage{algorithmic}
\usepackage{graphicx}
\usepackage{algorithm,algorithmic}
\usepackage{hyperref}
\hypersetup{hidelinks=true}
\usepackage{textcomp}
\usepackage{cite}
\usepackage{tikz}
\usepackage{xcolor}
\usepackage{verbatim}
\usetikzlibrary{arrows}

\def\BibTeX{{\rm B\kern-.05em{\sc i\kern-.025em b}\kern-.08em
    T\kern-.1667em\lower.7ex\hbox{E}\kern-.125emX}}
\DeclareMathOperator*{\argmin}{arg\,min}
\newtheorem{definition}{Definition}
\newtheorem{corollary}{Corollary}
\newtheorem{remark}{Remark}

\newtheorem{assumption}{Assumption}

\newtheorem{proposition}{Proposition}
\newtheorem{ex}{Example}
\newtheorem{theorem}{Theorem}
\newcommand{\R}{\mathbb{R}}
\newcommand{\ba}{\begin{array}}
\newcommand{\ea}{\end{array}}
\newcommand{\be}{\begin{equation}}
\newcommand{\ee}{\end{equation}}

\newcommand{\E}{\mathbb{E}}
\renewcommand{\P}{\mathbb{P}}
\newcommand{\ds}{\displaystyle}
\newcommand{\mc}{\mathcal}
\newcommand{\supp}{\text{supp}}

\newcommand{\mb}{\mathbf}
\newcommand{\1}{\mathbf 1}
\newcommand{\0}{\mathbf 0}
\newcommand{\ups}{\upsilon}

\begin{document}

\title{On Optimality of Private Information in Bayesian Routing Games}
\author{Alexia Ambrogio, \IEEEmembership{Student Member, IEEE}, Leonardo Cianfanelli, \IEEEmembership{Member, IEEE}, and Giacomo Como, \IEEEmembership{Member, IEEE}
\thanks{Preliminary results for the special cases of transportation networks consisting of two parallel links and $n$ parallel links were presented at \cite{cianfanelli2023information} and \cite{ambrogio2024information}, respectively.}
\thanks{This work was partially supported by the Research Project PRIN 2022 ``Extracting Essential Information and Dynamics from Complex Networks'' (Grant Agreement number 2022MBC2EZ) funded by the Italian Ministry of University and Research.}}

\maketitle

\begin{abstract}
We study an information design problem in transportation networks, in the presence of a random state that affects the travel times on the links. An omniscient system planner ---aiming at reducing congestion--- observes the network state realization and sends private messages to the users ---who share a common prior on the network state but do not observe it directly--- in order to nudge them towards a socially desirable behavior. The desired effect of these private signals is to correlate the users' selfish decisions with the network state and align the resulting Bayesian Wardrop equilibrium with the system optimum flow. Our main contribution is to provide sufficient and necessary conditions under which optimality may be achieved by a fair private signal policy in transportation networks with injective link-path incidence matrix and affine travel time functions.
\end{abstract}


\section{Introduction}
\label{sec:intro}

Reducing congestion in transportation networks is a main goal to decrease pollution and improve the quality of life in urban areas. To address this challenge, it is important to understand how individuals make routing decisions in response to the network structure, the information they have on the state of the network, and the other users' behavior. Such understanding lays the foundations for the system planner to effectively design influence policies to stir the users' decisions and nudge them towards socially desirable behaviors. In fact, recent technologies have increased the range of opportunities and, in particular, they have opened up the possibility to design personalized routing recommendation systems.  

Mathematical models that capture the strategic user behavior, such as routing games, provide a powerful tool to describe strategic decision-making and plan interventions \cite{wardrop1952road,beckmann1956studies}. Routing games model the transportation network by a directed multigraph, endowing each link of the network with a delay function that models the travel time across the road as a non-decreasing function of the flow that captures congestion effects. The large number of users is modeled by a non-atomic set of players, in the spirit of population games \cite{sandholm2010population} and the user behavior is captured by the notion of Wardrop equilibrium, a flow such that no one user has a strict incentive to unilaterally deviate to a different route.  

Since Wardrop equilibria arise from users trying to minimize their personal cost, instead of coordinating with other users to minimize the system cost, they are typically suboptimal in terms of total travel time spent on the network. Such inefficiency may be measured by the Price of Anarchy, defined as the ratio between the total travel time at the worst equilibrium and the one at the optimal flow \cite{roughgarden2002price}. Several approaches have been adopted in the literature to optimize the quality of the equilibria. Two classical ones involve interventions on the network structure (known as network design problem \cite{cianfanelli2023optimal,farahani2013review,roughgarden2006severity,marcotte1986network}) and the design of incentive mechanisms such as tolls, which instead influence the user behavior without modifying the infrastructure \cite{cole2006much,Como.Maggistro:22,fleischer2004tolls}.

In the past years, Traveler Information Systems have been proposed as additional solution to reduce congestion, with the implicit assumption that providing users with real-time information would bring beneficial effects \cite{srinivasan2000modeling}. However, it is a well established fact in the literature that providing users with more information can in fact hurt the system performance and that more informed users may be outperformed by non-informed ones  \cite{acemoglu2018informational,wu2021value}. 

Based on these observations, several works in the literature have started using information design theory to understand what type of information should be given to the users in order to maximize the system performance \cite{das2017reducing,tavafoghi2017informational}. The setting of information design is the following. There is a random network state (modeling, e.g., accidents, uncertain whether conditions, roadworks) affecting the delay functions, and a central planner who has access to its realization and has an informational advantage with respect to the users, as the latter know the prior distribution of the network state, but do not observe its realization. Based on the observation of the network state, the planner discloses some information to the users according to a certain rule. Since public information, e.g., conveyed via road signals, is known to be inefficient in routing games \cite{massicot2022competitive,das2017reducing,tavafoghi2017informational,zhu2022information}, we focus on private signaling sent through the users' personal devices. In this context, the planner selects a finite set of messages and a private signal policy, which associates to each player a message that depends on the network state realization and possibly on other random variables. By doing so, the planner reshapes the posterior beliefs of the users, who select their paths based on the message that they receive, resulting in a so-called \emph{Bayesian Wardrop equilibrium} (BWE), a flow that is correlated with the network state through the message that the users receive. The ultimate goal of the planner is then to design a set of messages and the private signal policy whose equilibrium achieves the minimum cost.

%The role of information is widely studied in the literature of routing games. \cite{kamenica2011bayesian} analyzes the optimal message between the central planner and one single user. \cite{acemoglu2018informational} shows that the travel times may decrease when a fraction of users is not aware of some roads of the network, while \cite{wu2017informational,wu2021value} show that too much information may hurt the system. In general, \cite{das2017reducing} and \cite{massicot2022competitive} prove that public information policies are in general suboptimal with respect to private ones, therefore in this paper we consider private information provision, motivated also by the fact that nowadays users have access to information through their personal devices, for instance as smartphones.

%%%% cosa si è fatto e si fa in letteratura
In the literature, information design in Bayesian routing games have been studied along several directions.
The problem has been first proposed in \cite{das2017reducing} and analyzed in \cite{tavafoghi2017informational} and \cite{wu2019information} on two parallel links with affine delay functions.
Dynamical information provision is studied in \cite{dughmi2016algorithmic,meigs2020optimal,tavafoghi2019strategic, ouyang2024anapproach}. The information design problem when only some users have access to information is studied in
\cite{zhu2022information} from an algorithmic perspective with polynomial delay functions and arbitrary network topologies, while \cite{doval2024constrained} provides a solution strategy to solve a constrained information design problem. Other papers consider limited rationality as in \cite{gould2023rationality,feng2024rationality}.
Bilateral information asymmetry between the central planner and the users has been considered in \cite{sezer2023robust}, while the co-presence of two populations with different access to the information has been studied in \cite{liu2016effects}. 
Moreover, \cite{massicot2019public} characterize optimal public signaling and \cite{massicot2024OnNetwork} analyzes boundedly rational user equilibria in a context of public signals. The design of information with unknown demand is studied in \cite{griesbach2024information}, while recently the coupling between information and tolls was considered in \cite{ferguson2024information}. 

In this paper, we focus on \emph{deterministic and fair} private signal policies, namely, such that the probability of receiving each message is the same for all agents traveling from the same origin node the same destination node and such that the fraction of users that receive each message is a deterministic function (called \emph{signaling rule}) of the network state realization. A more general setting that generalizes deterministic policies involves an additional randomization step (sometimes called \emph{garbling}), that is, for every network state realization, the planner may choose at random the fraction of users that receive each message \cite{massicot2022competitive,koessler2024correlated,massicot2025strategic}, but this step is unnecessary for our purposes, as discussed later.
We show that, for arbitrary network topologies and delay functions, the BWE corresponding to deterministic fair private signal policies depends on the signal policy only via the signaling rule, and prove that an analog of the revelation principle for finite strategic games (see \cite{bergemann2019information,myerson1982optimal}) holds true in our non-atomic setting. The implication of the revelation principle is that it is without loss of generality for the planner to restrict to direct obedient signaling rules, defined as signaling rules such that messages coincide with path recommendations that are convenient to follow for the users (Theorem \ref{rv}). We then build on this result to establish sufficient and necessary conditions under which the planner can induce the system-optimum flow as BWE for deterministic fair private signal policy. These conditions hold true for networks with affine delay functions and such that they admit a bijection between link flows and path flows, generalizing the results of \cite{cianfanelli2023information,ambrogio2024information} to non-parallel networks (Theorem \ref{thm:opt_gen}). Since these conditions are hard to interpret, we then provide more intuitive sufficient conditions for optimality, that become more restrictive as the class of network topology becomes more general.

The paper is organized as follows.
In Section \ref{sec:2} we define the transportation network, the notion of information and formulate the problem. Section \ref{sec:pub_policy} establishes some preliminary considerations and motivates the problem. Section \ref{sec:results} provides our main results. Section \ref{sec:examples} illustrates some examples, while Section \ref{sec:conclusion} summarizes the contribution and future research lines.

\subsection{Notation}
Given a finite set $\mc A$, we let $\R^{\mc A}$ denote the set of real valued vectors, whose entries are indexed by the elements of $\mc A$. Given a vector $x$, we let $x'$ denote its transpose and $[x]$ denote the diagonal matrix with $[x]_{ii} = x_i$. $\mathbf 1$, $\mathbf 0$,  $\mathbf I$ and $\delta^{(i)}$ denote the vector of all ones, the vector of all zeros, the identity matrix and the vector with $1$ in the $i$-th entry $\delta_i^{(i)}$ and $0$ in all other positions, whose size may be deduced from the context.
%We let $\Delta_{\mathcal{A}}$ denote the simplex on $\mathcal{A}$, i.e.
%$$
%\Delta_{\mathcal{A}} = \{x \in \R_+^{\mathcal{A}}: \mathbf 1'x = 1\}.
%$$
Given a random variable $X$ in $\R^n$, its support is the set of values that the random variable can take, i.e.,
$$
\supp(X) = \{x \in \R^n: \P(B(x,\epsilon))>0 \ \forall \epsilon > 0\},
$$
where $B(x,\epsilon)$ is the open ball centered in $x$ with radius $\epsilon$.

%Given an element $x$ in $\R$, we let
%$$
%[x]_0^1=\max\{0,\min\{x,1\}\} = \begin{cases}
%	0 \ &\text{if} \ x < 0 \\
%	x \ &\text{if} \ x \in [0,1] \\
%	1 \ &\text{if} \ x > 1. \\
%\end{cases}
%$$
%Given an element $x$ in $\R^+$, we let
%$$
%x_{mod \ 1}=x-\lfloor x \rfloor
%$$
%where $\lfloor x \rfloor = \max\{k \in \Z : k \leq x\}$.
%Then, given a set $X$ we define $X_{mod \ 1}$ the union of all the elements of $X$ to whom the operation $mod \ 1$ has been applied.

\section{Problem formulation and preliminary results}\label{sec:2}

%\begin{figure}[!t]
%\centerline{\includegraphics[width=\columnwidth]{}}
%\caption{}
%\label{fig1}
%\end{figure}

In this section, we define the transportation network and the notions of information and of Bayesian Wardrop equilibrium. We then establish some preliminary results, formulate our problem and provide a motivating example.

\subsection{Transportation network model}
\label{sec:structure}

The \emph{transportation network} is modeled as a directed multigraph $\mathcal{G}=\left(\mathcal{V},\mathcal{E}, \sigma, \xi \right)$, where $\mathcal{V}$ is a non-empty finite set of nodes, $\mc E$ is a finite set of links, and $\sigma, \xi : \mathcal{E} \rightarrow \mathcal{V}$ are two maps that associate to every link $e$ in $\mc E$ its tail node $\sigma(e)$  and head node $\xi(e)$, respectively. Throughout the paper, we shall assume that there are no self-loops, i.e., that $\sigma(e)\ne\xi(e)$ for every link $e$ in $\mc E$. 
%The transportation network topology can then be equivalently characterized by the node-link incidence matrix $B$ in $\R^{\mathcal V \times \mathcal E}$, with entries 
%$$B_{ie} =\left\{\ba{lcl} 1 &\se &\sigma(e)=i\\-1&\se&\xi(e)=i\\0&\se&\sigma(e)\ne i\ne\xi(e)\,,\ea\right.$$ 
%for every node $i$ in $\mc V$ and link $e$ in $\mc E$ \tcr{non sono sicuro che ci serva $B$}. 

For two distinct nodes $o\ne d$ in $\mc V$, and a positive integer $l$, a length-$l$ $(o,d)$-path is an $l$-tuple of links $\gamma = \left(e_1,...,e_l\right)$ in $\mc E^l$ such that 
$\sigma\left(e_1\right)=o$, $\xi\left(e_l\right)=d$, 
$$\xi\left(e_{i}\right)=\sigma\left(e_{i+1}\right)\,,\qquad 1\le i<l\,,$$ 
and 
$$\sigma\left(e_i\right)\ne\sigma(e_j)\,,\qquad 1\le i<j\le l\,.$$
The set of all $(o,d)$-paths of any length $l\ge1$ is denoted by $\Gamma_{od}$. Notice that the set $\Gamma_{od}$ is necessarily finite.
Node $d$ is said to be reachable in $\mc G$ from node $o$ if there exists at least one $(o,d)$-path $\gamma$, i.e., if $\Gamma_{od}$ is nonempty. We shall denote the set of all paths in $\mc G$ from any origin node $o$ in $\mc V$ to any destination $d$ in $\mc V\setminus\{o\}$ by $\Gamma=\bigcup_{o\ne d}\Gamma_{od}$. 
We then define the link-path incidence matrix $A$ in $\{0,1\}^{\mathcal E \times \Gamma}$, with entries $A_{e\gamma} = 1$ if link $e$ belongs to path $\gamma$ and $A_{e\gamma} = 0$ otherwise.

A feasible throughput is a nonnegative square matrix $\upsilon$ in $\R_+^{\mc V\times\mc V}$ such that $\upsilon_{ii}=0$ for every $i$ in $\mc V$, i.e., $\upsilon$ has zero diagonal, and $\upsilon_{od}=0$ for every two nodes $o$ and $d$ such that $d$ is not reachable from $o$ in $\mc G$. For a feasible throughput $\upsilon$, the set of compatible path flow assignments is 
$$
\mc Z_{\upsilon} = \left\{z \in \R^{\Gamma}_+: \sum\nolimits_{\gamma \in \Gamma_{od}} z_\gamma = v_\gamma\,, \ \forall o,d\in\mc V\right\}\,,
$$ 
while 
$$
\mc F_{\upsilon} = \left\{f \in \R^\mc E_+: f = Az\,, \ z \in \mc Z_{\upsilon}\right\}\,,
$$
denotes the set of compatible link flow vectors.
%The exogeneous inflow $\nu$ in $\R^{\mc V}$ is then defined by $$\nu = \sum_{\gamma \in \Gamma} v_{\gamma}(\delta^{(o_\gamma)}-\delta^{(d_\gamma)}),$$ and the set of flow vectors that can be induced by the exogeneous inflow is 
%$$
% \mc F_{\ups} = \{f \in \R^{\mathcal E}_+: Bf=\nu\}.
%$$
%be the set of unitary $o$-$d$ flows, where a flow describes how users are distributed across links.

The \emph{network state} is described by a random vector $\Theta$ defined on a complete probability space $(\Omega, \mathcal{A}, \P)$ and taking values in $\R_+^{\mathcal E}$. We shall assume that the entries of the network state have bounded first and second moment, i.e., 
$$\E[\Theta_e]<+\infty\,,\qquad \E[\Theta_e^2]<+\infty\,, \qquad \forall e \in \mc E\,.$$ 
We do not make any assumption on the independence of the entries of $\Theta$, while the case where they are independent can be recovered from our framework as a special case.  

Every link $e$ in $\mc E$ is endowed with a measurable delay function $\tau_e : \R_+ \times \R_+ \to \R_+$ that returns the travel time $\tau_e(\theta_e,f_e)$ on that link as a function of the realization of the $e$-th entry $\theta_e$ of the network state and of the flow $f_e$ on that link. In order to consider congestion effects we shall assume that $\tau_e(\theta_e, f_e)$ is continuously differentiable, strictly increasing
% per avere unicità equilibrio 
and convex in $f_e$ for every $\theta_e$.

A network flow is a function $f : \R^{\mc E}_+ \to  \mc F_{\ups}$ that assigns to every realization of the network state a link flow in $ \mc F_{\ups}$.
The \emph{system cost} $C(f)$ of a network flow $f : \R^{\mc E}_+ \to  \mc F_{\ups}$ is the overall expected travel time spent on the network, i.e.,
\begin{equation}\label{system_cost}
\begin{array}{rcl}
\ds C(f) &= &\ds \E\Big[\sum_{e \in \mathcal E} f_e(\Theta) \tau_e (\Theta_e, f_e(\Theta))\Big] \\[14pt]
&=&\ds \int_{\R_{+}^{\mc E}} \sum_{e \in \mathcal E} f_e(\theta) \tau_e (\theta_e, f_e(\theta)) d\P(\theta)\,.
\end{array} 
\end{equation}
It is then natural to define the \emph{system-optimum flow} as follows. 
Let
\be\label{eq:system_opt_full}
f^*(\theta) = \argmin_{x \in \mc F_\ups} \sum_{e \in \mc E} x_e \tau_e(\theta_e,x_e)\,, \quad \forall \theta \in \R_+^{\mc E}\,,
\ee 
%is the vector in $\mc F$ that minimizes $\sum_{e \in \mathcal E} f_e\left(\theta\right) \tau_e (\theta_e, f_e(\theta))$, 
whose uniqueness follows from the strict monotonicity and convexity of the delay functions with respect to the link flow.
Then, observe that 
$$
	C(f^*) = \min_{f: \R_+^{\mc E} \to  \mc F_{\ups}} C(f)\,.
$$
%\begin{remark}
%Observe that the system-optimum $f^*$ is unique and independent of the prior distribution of the network state. Indeed, 
%For every $\theta$ in $\R_+^{\mc E}$,
%$$
%f^*(\theta) = \argmin_{x \in \mc F} \sum_{e \in \mc E} x_e \tau_e(\theta_e,x_e)\,,
%$$
%%is the vector in $\mc F$ that minimizes $\sum_{e \in \mathcal E} f_e\left(\theta\right) \tau_e (\theta_e, f_e(\theta))$, 
%whose uniqueness follows from the strict monotonicity and convexity of the delay functions with respect to the link flow.
%\end{remark}
The system-optimum can be interpreted as the network flow that an omniscient planner would impose on the network if the user behavior could be chosen in a centralized manner as a function of the network state. However, the transportation network users are strategic and aim to minimize their own expected travel cost given the information they have on the network state. In the next sections, we shall illustrate how information determines the user behavior and therefore the resulting equilibrium flow.

\subsection{Information structure}
\label{sec:behavior}
We assume that the omniscient planner can observe the realization of the network state $\Theta$. Based on this observation, she sends private signals to the users in order to influence their belief on the network state and the resulting equilibrium flow, with the goal of minimizing the system cost at the equilibrium. 

First, for a given feasible throughput $\ups$, we introduce the \emph{user set}
$$
\mc U_{\ups} = \{(o,d,u): o,d \in \mc V\,, u \in [0,\ups_{od})\}\,.
$$
Every user $(o,d,u)$ in $\mc U_{\ups}$ is then identified by its origin node $o$, destination node $d$ and a point $u$ in the interval $[0,\ups_{od})$. Observe that, whenever $\ups_{od}=0$ ---hence, in particular, when $d$ is not reachable from $o$--- there are no users with origin node $o$ and destination node $d$. On the other hand, when $d$ is reachable from $o$, there is a continuum of users with origin node $o$ and destination node $d$ of Lebesgue measure $\ups_{od}$.

\begin{definition}
Consider a finite \emph{message set} $\mc M$. 
%\begin{itemize}
%	\item the \emph{user set} $$\mc U = \prod_{o,d \in \mc V} \mc U_{od}\,,$$
%	where $\mc U_{od} = [0,\ups_{od}]$ is the set of users traveling from node $o$ to node $d$;
%	\item a finite \emph{message set} $\mc M$;
%\item 
A \emph{(private) signal policy} is a map $S: \mc U_{\ups}\times \Omega \to \mc M$ that assigns to each user a message that depends on the network state and possibly on other random variables. A private signal policy $S$ is referred to as \emph{fair} if, for every realization $\theta$ of the network state $\Theta$ in $\supp(\Theta)$, the conditional distribution of the message received by any two users with the same origin/destination pair is the same, i.e., 
$$\P(S((o,d,u),\omega)\!=\!m|\Theta\!=\!\theta)\!=\!\P(S((o,d,u'),\omega)\!=\!m|\Theta\!=\!\theta)\,,$$
for every $o$ and $d$ in $\mc V$ and $u,u'$ in $[0,\ups_{od})$. 
\end{definition}

Throughout this paper, we shall always assume that the messages sent to the users are drawn according to a fair signal policy.
To every fair signal policy  $S: \mc U_{\ups} \times \Omega \to \mc M$, we can then associate a map $\pi : \R_+^{\mc E} \to \R_+^{\mc V \times \mc V \times \mc M}$, to be referred as the associated \emph{(signaling) rule}, defined by 
\be\label{eq:signaling_rule} \pi_m^{od}(\theta) =\P(S((o,d,u),\omega)=m|\Theta=\theta)\,,\ee
i.e., 
$\pi_m^{od}(\theta)$ is the conditional probability that any user with origin node $o$ and destination node $d$ receives message $m$ given that $\Theta = \theta$. The set of signaling rules associated to all fair signal policies with given set of messages $\mc M$ is defined by
$$
\Pi_{\mc M}\!=\! \left\{\pi: \R_+^{\mc E} \to \R_+^{\mc V \times \mc V \times \mc M}:\!\! \sum_{m \in \mc M}\! \pi_m^{od}(\theta) \! = \! 1\,, \ \forall o,d \in \mc V\right\}.
$$ 

We assume that all users know the prior distribution of the network state, the signaling rule associated to the fair signal policy, and their private message, but cannot observe the realization of the random variables. Once a user receives her message, she updates her belief on the network state as established by the following result. 
\begin{proposition}\label{prp:bayes}
	Consider a message set $\mc M$ and a fair signal policy $S$ with associated signaling rule $\pi$ in $\Pi_\mc M$. Then, for every two nodes $o,d$ in $\mc V$, the posterior belief on the network state of every user $(o,d,u)$ that receives message $m$ is
	\be\label{eq:posterior_prp}
	d\P_m^{od}(\theta) = \ds \frac{\pi^{od}_m(\theta) d\P(\theta)}{\int_{\R_+^{\mc E}} \pi^{od}_m(\theta') d\P(\theta')}\,,
	\ee
for every $m$ in $\mc M$ such that $\int_{\R_+^{\mc E}} \pi_m^{od}(\theta)d\P(\theta)>0$.
\end{proposition}
\begin{proof}
	The statement follows from \eqref{eq:signaling_rule} and 
	from Bayes' theorem.
\end{proof}\medskip

While a fair signal policy $S$ induces a unique signaling rule $\pi$, many signal policies $S: \mc U_{\ups}\times \Omega \to \mc M$ are compatible with a rule $\pi$ in $\Pi_{\mc M}$.
\begin{remark}\label{rem:posterior}
	Note by Proposition \ref{prp:bayes} that, given a fair signal policy $S$, the posterior beliefs depends on $S$ via the associated signaling rule $\pi$. Hence, all fair signal policies $S$ with same corresponding signaling rule $\pi$ induce the same posterior beliefs.
\end{remark}

One canonical way to construct a fair signal policy $S$ given a signaling rule $\pi$ is the following.

\begin{ex}\label{ex:signal}
Consider a finite set of messages $\mc M$, a map $\pi$ in $\Pi_\mc M$ and a random matrix $\Psi$ defined on the probability space $(\Omega, \mathcal{A}, \P)$, independent of the network state $\Theta$ and uniformly distributed on the hypercube $[0,1]^{\mathcal V \times \mc V}$. Then, let the message for each user $(o,d,u)$ in $\mc U_\ups$ be $S((o,d,u),(\theta,\psi)) = m$ if and only if
\begin{equation}\label{eq:message}
	\sum_{j=1}^{m-1}\pi_j^{od}(\theta) \le \psi_{od}+\frac{u}{\ups_{od}}-\left\lfloor \psi_{od}+\frac{u}{\ups_{od}} \right\rfloor < \sum_{j=1}^m \pi_j^{od}(\theta)\,.
\end{equation} 
\end{ex}\smallskip

A property of the fair private signal policy defined in Example \ref{ex:signal} is that the fraction of users that receive each message is a deterministic function of the network state and does not depend on any other randomness. The class of fair private signal policies with this property will be referred to as \emph{deterministic}. However, we remark that the one defined in Example \ref{ex:signal} is not the only deterministic fair private signal policy with associated signaling rule $\pi$. Moreover, not all fair private signaling rules are deterministic. In fact, in principle the definition of signaling rule given in \eqref{eq:signaling_rule} allows for a further randomization, that is, for every realization of the network state, the fraction of users that receive each message might depend on additional random variables. These policies are sometimes referred to in the literature as \emph{garbled policies} \cite{massicot2022competitive,massicot2025strategic,koessler2024correlated,koessler2025full}. Throughout this paper, we restrict our analysis to deterministic fair signal policies, as a further randomization is unnecessary for our purposes, as discussed later.
%assume that the planner designs a signaling rule $\pi$ and that the associated fair private signal policy $S$ is defined by Example \ref{ex:signal} accordingly and is therefore deterministic.

\subsection{Bayesian user equilibrium}
%Let $\P^{\gamma}_m$ denote the posterior belief on the network state $\Theta$ of $\gamma$-users that receive message $m$.

%\begin{remark}
%The key for our results is that, given two nodes $o,d$ in $\mc V$, all users $(o,d,u)$ in $\mc U$ that receive message $m$ have equal posterior belief \eqref{eq:posterior_prp}. This in turn follows from the fact that the signal is fair, namely, for every two nodes $o,d$ in $\mc V$, all users in $(o,d,u)$ in $\mc U$ have equal probability of receiving each message $m$, as established by Proposition \ref{prp:bayes}(i). In fact, a consequence of Proposition \ref{prp:bayes}(ii) is that the outcome of a private signal depends solely on $\pi$, which determines the mass of users that receive each message and the posterior belief that the users form about the network state. For this reason, in the rest of the paper we shall focus on the rule $\pi$ to indicate the object that the planner aims to optimize. However, we remark that this is not the only fair private signal that induces Lebesgue measure $\pi$ (cf. Remark \ref{remark:Lebesgue}) and that our results hold true for all private signals that satisfy Proposition \ref{prp:bayes}(i).
%\end{remark}

Proposition \ref{prp:bayes} illustrates how the users update their beliefs on the network state after receiving a message drawn according to a fair signal policy with known associated signaling rule. After forming their posterior belief on the network state, the users choose a path with minimum expected cost. The cost of a path is defined as sum of the delay functions of the links that compose the path, i.e.,
\begin{equation}\label{eq:cost_path}
	c_\gamma (\theta, f) = \sum_{e \in \mathcal E} A_{e\gamma} \tau_e(\theta_e, f_e)\,, \! \quad \forall \theta \in \R_+^{\mc E}, \ \forall \gamma \in \Gamma.
\end{equation}
To formalize the notion of equilibrium flow, we introduce $y$ in $\R_+^{\mc V \times \mc V \times \Gamma \times \mathcal M}$, whose element $y_{\gamma m}^{od}$ indicates, for every two nodes $o,d$ in $\mc V$, the fraction of users with origin node $o$ and destination node $d$ that choose path $\gamma$ conditioned on the fact that they received message $m$. By construction, $y$ must satisfy the two constraints
\be\label{eq:con:y}
\ba{rcl}
\ds y^{od}_{\gamma m} & = & 0\,, \quad \ds \forall \gamma \notin \Gamma_{od}\,, \\[8pt]
\ds \sum_{\gamma \in \Gamma} y^{od}_{\gamma m} & = & \ds 1\,,
\ea\ee 
for every two nodes $o,d$ in $\mc V$ and message $m$ in $\mc M$. 

Note that, given a deterministic fair signal policy $S$ with associated signaling rule $\pi$, and given $y$, the total mass of users that travel on path $\gamma$ in $\Gamma_{od}$ is $\sum_{m \in \mathcal M} y_{\gamma m}^{od} \pi_m^{od}(\theta) \ups_{od}$, which depends on $S$ only via the associated rule $\pi$. Theferore, the corresponding network flow is
\begin{equation}\label{eq:flow_y}
f_e^{\pi,y}(\theta) = \sum_{o,d \in \mc V} \sum_{\gamma \in \Gamma_{od}} A_{e\gamma} \sum_{m \in \mathcal M} y_{\gamma m}^{od} \pi_m^{od}(\theta) \ups_{od}\,,
\end{equation}
for every $e$ in $\mc E$ and $\theta$ in $\R_+^{\mc E}$.
%More compactly,
%\begin{equation*}
%f^{\pi,y}(\theta)=Ay\pi(\theta)\,.
%\end{equation*}
%\tcr{questa non vale più, la posso riutilizzare solo nel caso singola $(o,d)$ con throughput unitario}
\begin{remark}\label{rem:flow}
Given $y$ and a deterministic fair signal policy $S$ with corresponding signaling rule $\pi$, the flow depends on $S$ only by $\pi$. Due to this observation and to Remark \ref{rem:posterior}, one can identify deterministic fair signal policies with the corresponding signaling rule. It is important to remark that for a signaling rule $\pi$ in $\Pi_\mc M$, a compatible deterministic fair signal policy $S$ may always be constructed following the procedure detailed in Example \ref{ex:signal}.
\end{remark}

We are now ready to define the notion of equilibrium flow.

%The condition of equilibrium is a situation in which users that have received message $m$ decide to go on path $i$ if and only if the posterior expected value of the cost on path $i$ is the lowest of the network with respect to the cost of all the other paths, i.e.
%\begin{equation*}
%\mathbb E_m[c_i(\Theta, f^{\pi,y}(\Theta))] \leq \mathbb E_m[c_j(\Theta, f^{\pi,y}(\Theta))] \quad \forall j \in \mathcal P
%\end{equation*}
%where, using Lemma \ref{bayes},
%\begin{equation*}
%\mathbb E_m[c_i(\Theta, f^{\pi,y}(\Theta))] = \frac{\int_{\mathbf{\Theta}} c_i\left(\theta, f^{\pi,y}\left(\theta\right)\right) \pi_m\left(\theta\right) d\mathbb{P}(\theta)}{\int_{\mathbf{\Theta}}\pi_m\left(\omega\right) d\P(\omega)}.
%\end{equation*}

\begin{definition}[Bayesian Wardrop equilibrium]
\label{def:bue}
Given a set of messages $\mc M$ and a signaling rule $\pi$ in $\Pi_{\mc M}$, a network flow $f^{\pi,y}$ \eqref{eq:flow_y} is a Bayesian Wardrop equilibrium (BWE) if, for every two nodes $o,d$ in $\mc V$, $i$ in $\Gamma_{od}$ and $m$ in $\mc M$ such that $y_{im}^{od}>0$, the inequality
\begin{equation}\label{eq:flow_equilibrium}
\int_{\R_+^{\mc E}} (c_i(\theta, f^{\pi,y}(\theta))-c_j(\theta,f^{\pi,y}(\theta)) \pi_m^{od} (\theta) d\P(\theta) \leq 0
\end{equation}
holds true for every $j$ in $\Gamma_{od}$.
\end{definition}
%\begin{remark}
%Note that the BWE equilibrium associated to a deterministic private signal policy $S$ depends on $S$ only through the associated signaling rule $\pi$.
%\end{remark}

The definition of BWE has the following interpretation. If $\int_{\R_+^{\mc E}} \pi_m^{od}(\theta) d\P(\theta) = 0$, namely, if almost no user $(o,d,u)$ receives message $m$, then \eqref{eq:flow_equilibrium} is trivially satisfied. If instead $\int_{\R_+^{\mc E}} \pi_m^{od}(\theta) d\P(\theta) > 0$, then \eqref{eq:flow_equilibrium} is equivalent by Proposition \ref{prp:bayes} to
%$$
%\int_{\R_+^{\mc E}} (c_i(\theta, f^{\pi,y}(\theta))-c_j(\theta,f^{\pi,y}(\theta)) d\P_m^{\gamma}(\theta) \leq 0\
%$$
requiring that, if some users $(o,d,u)$ receive message $m$ and select path $i$ in $\Gamma_{od}$, then path $i$ has to be no worse than any alternative path $j$ in $\Gamma_{od}$ according to the posterior belief of users $(o,d,u)$ that received message $m$.

The next result characterizes BWEs, providing a sufficient and necessary condition on $y$ under which a network flow $f^{\pi,y}$ is a BWE for a rule~$\pi$.
\begin{proposition}\label{prp:pot}
	Given a rule $\pi$, a network flow $f^{\pi,y}$
	is a BWE if and only if
	\begin{equation}\label{eq:y} 
		y \in \argmin_{x\in \R_+^{\mc V \times \mc V \times \Gamma \times \mathcal M}: \eqref{eq:con:y}} \Phi_\pi \left(x\right),
	\end{equation}
	where
$$
		\Phi_\pi(y) = \int_{\R_+^\mc E}\sum_{e \in \mathcal E} \int_0^{f_e^{\pi,y}(\theta)} \tau_e (\theta_e,s)ds d\P(\theta)\,.
$$
Moreover, the BWE is unique for every rule $\pi$.
\end{proposition}

\begin{proof}
The proof is deferred to Appendix \ref{app:A}.
\end{proof}\medskip

\begin{remark}
Notice that $y$ can be interpreted as the action distribution of a weighted potential population game with multiple populations that differ in the received message $m$ and in the action sets $\Gamma_{od}$, since
\be\label{eq:pop_game}
\begin{aligned}
\!\!\! \frac{\partial {\Phi_\pi(y)}}{\partial y_{\gamma m}^{od}} & = \int_{\R_+^{\mc E}} \sum_{e \in \mc E} \tau_e (\theta_e, f^{\pi,y}_e(\theta))A_{e\gamma} \pi_m^{od}(\theta) \ups_{od}d\P (\theta)\\[3pt]
& = \ups_{od} \int_{\R_+^{\mc E}} c_\gamma(\theta, f^{\pi,y}(\theta)) \pi_m^{od} (\theta) d\P(\theta) \,,
\end{aligned}
\ee
i.e., the partial derivative of the potential function with respect to $y_{\gamma m}^{od}$ is equal to the expected cost function of path $\gamma$ for users that travel from $o$ to $d$ and receive message $m$ times a weight $\ups_{od}$.
%, to where the equivalence follows from \eqref{eq:flow_y}
%The weights for the population of users that travel from $o$ to $d$ and receive message $m$ is given by $\ups_{od}$.
For more details on this class of population games, see \cite{sandholm2010population}.
%%Proposition \ref{prp:pot} in particular implies that $y$ is the Nash equilibrium of a population game that admits a potential function
\end{remark}

Observe that, although multiple $y$ may satisfy \eqref{eq:y} for a rule $\pi$, the BWE $f^{\pi,y}$ is always unique by Proposition~\ref{prp:pot}. For this reason, for convenience of notation, we shall sometimes denote by $f^\pi$ the BWE corresponding to rule $\pi$ and omit $y$ when unnecessary.
%\tcr{potrei definire il costo di una policy come il costo della peggior policy}

Since the equilibrium flow $f^{\pi}$ is unique for every rule $\pi$, it is natural to define the \emph{cost of a rule} $\pi$ as the cost of the corresponding BWE $C(f^{\pi})$.

%\begin{remark}
%In this paper we focus on signal policies such that the fraction of users that receive each message is a deterministic function of the realization of the network state, as the one defined in Example \ref{ex:signal}. As a consequence of this and of Proposition \ref{prp:bayes}, the outcome of a fair private signal policy $S$ depends solely on the associated signaling rule $\pi$, which determines the mass of users that receive each message and the posterior belief that the users form about the network state.
%\end{remark}
%\tcr{toglierei il price of anarchy oppure lo terrei solo a fini simulativi}
%Moreover, we define the price of anarchy ($PoA$) of a policy $\pi$ as the ratio between the cost of the equilibrium $f^\pi$ and the optimal cost $f^*$, i.e.,
%\begin{equation}\label{eq:PoA}
%PoA (\pi) = \frac{C(f^\pi)}{C(f^*)} \ge 1.
%\end{equation}
%In other words, in the context of information design the price of anarchy measures how suboptimal an information policy is with respect to the system optimum.
%

\subsection{Problem formulation}
\label{sec:id_problem}

%The planner aims to jointly design the optimal set of messages $\mc M$ and deterministic fair private signal policy $S$ that minimize the system cost. 
In this paper, we aim to find conditions on the network topology, on the delay functions, and on the distribution of the network state under which optimality may be achieved by a deterministic fair signal policy $S$. In other words, we ask when the system-optimum flow may be obtained as a Bayesian Wardrop equilibrium for a certain signaling rule $\pi$, i.e., whether there exists a rule $\pi$ such that $f^{\pi} = f^*$. Then, a signal policy $S$ that complies with the rule $\pi$ may be constructed using the procedure in Example \ref{ex:signal}. 

In general, imposing that the signal policy is deterministic cancels out a degree of freedom that may be useful for the planner to improve the equilibrium cost (cf. \cite[Example 6]{koessler2022information}\footnote{This working paper has been published \cite{koessler2024correlated}, but the final version does not contain this example},\cite{massicot2022competitive}). However, we shall see that this restriction is without loss of generality when the question is whether optimality may be achieved by information.
%Remark \ref{rem:flow} allows to identify $S$ with the corresponding rule $\pi$. The information design problem is
%\begin{equation}\label{probpre}
%	(\mathcal{M}^*,\pi^*) \in \argmin_{\mathcal{M},\pi \in \Pi_\mc M} C(f^\pi). \\[6pt]
%\end{equation}
%%\end{problem}
%%\begin{remark}\label{remark:convex1}

%This problem is in general hard to solve for two main reasons. First, it involves the design of the optimal set of messages, which in principle may be every finite set of messages. Second, t

%following subproblem. 
%\begin{problem}\label{problem:1}
%Find conditions on $(\mc G,\tau,\Omega,\mc A,\P)$ (\tcr{non mi piace com'è scritto ma voglio avere un ambiente formale col problema per metterlo in luce (anche un ambiente equazione può andar bene)}) under which the optimal rule $\pi^*$ achieves the optimal cost, or, equivalently, $f^{\pi^*} = f^*$. 
%\end{problem}

\section{Preliminary considerations}
\label{sec:pub_policy}
An important class of rules is the class of \emph{public rules}. This class of rules is characterized by the fact that all users receive the same message, i.e. for every $\theta$ in $\text{Supp}(\Theta)$ there exists a message $\tilde m(\theta)$ such that $$
\pi^{od}_m(\theta) =
\left\{ \ba{ll}
 1 \quad & \text{if} \ m = \tilde m(\theta) \\
 0 & \text{otherwise}\,,
\ea
\quad \forall o,d \in \mc V\,.
\right.
$$
Two examples of public rules are \emph{no information} and \emph{full information}. This section illustrates these rules and motivates our problem by showing that they are in general suboptimal.
%The first one is when the policy $\pi$ gives \textit{no information} about the network state, i.e.  $\pi(\theta)$ is independent of $\theta$.
\subsection{No information}\label{sec:no}
The no information scenario can be casted in our framework as follows. There exists a message $m^*$ in $\mc M$ such that $\tilde m(\theta) = m^*$ for every $\theta$ in $\R_{+}^{\mc E}$, so that the planner sends to all users message $m^*$ independent of the network state realization. Let the no information rule be denoted by $\pi^{NI}$. In this case, the network flow \eqref{eq:flow_y} has entries 
\be\label{eq:f_no_info}
f_e^{\pi^{NI},y}(\theta)=\sum_{o,d \in \mc V} \sum_{\gamma \in \Gamma_{od}} A_{e\gamma} y_{\gamma m^*}^{od} \ups_{od}\,,
\ee 
which does not depend on the network state realization, since $\tilde m(\theta) = m^*$ does not. As a consequence of Proposition \ref{prp:pot}(ii), the posterior belief of all users coincides with their prior belief, i.e., $d\P_{m^*}^{od}(\theta) = d\P(\theta)$ for every $\theta$ in $\R_+^\mc E$ and for every two nodes $o,d$ such that $\ups_{od}>0$. Hence, the equilibrium condition in Definition \ref{def:bue} is equivalent to say that $f^{\pi^{NI},y}$ is a BWE if, for every path $i$ in $\Gamma_{od}$ such that $y_{im^*}^{od}>0$, it holds true
\begin{equation}\label{eq:no_info_eq}
	\mathbb{E}[c_i (\Theta,f^{\pi^{NI},y})]
	\le \mathbb{E}[c_j (\Theta,f^{\pi^{NI},y})]\,, \quad \forall j \in \Gamma_{od}\,.
\end{equation}
In plain words, if some users travel along path $i$ in $\Gamma_{od}$, then path $i$ has to be in expectation no worse than any alternative path $j$ in $\Gamma_{od}$ according to the prior belief, which coincides with posterior belief of all users. 

\subsection{Full information}\label{sec:full}
The full information scenario assumes that all users are informed of the realization of the network state. This scenario is formally well defined in our framework under the assumption that the network state $\Theta$ may take a finite number of values. In this case, $\mathcal M = \text{supp}(\Theta)$ and $\tilde m(\theta) = \theta$ for every $\theta$ in $\R_+^{\mc E}$, 
%and the full information policy $\pi^{FI}$ is defined by
%$$
%\pi_{m}^{\gamma}(\theta) =
%\left\{
%\ba{ll}
%1 \quad & \text{if} \ m = \theta \\
%0 \quad & \text{otherwise}\,,
%\ea
%\quad \forall \gamma \in \Gamma\,,
%\right.
%$$
namely, all users receive the same message, which reveals the exact network state realization $\theta$. 
Let the full information rule be denoted by $\pi^{FI}$.
As a consequence of Proposition \ref{prp:bayes}(ii), the posterior beliefs are
$$\P_m^{od}(\theta) =
\begin{cases}
1 \quad & \text{if} \ m = \theta \\
0 & \text{otherwise}\,,
\end{cases}
\quad \forall o,d \in \mc V, \ \forall m \in \mc M\,.
$$
In this case, the network flow \eqref{eq:flow_y} has entries
\be\label{eq:full_info_flow}
f_e^{\pi^{FI},y}(\theta)=\sum_{o,d \in \mc V} \sum_{\gamma \in \Gamma_{od}} A_{e\gamma} y^{od}_{\gamma\theta}\ups_{od}\,,
\ee
since the total flow over path $\gamma$ in $\Gamma_{od}$ when $\Theta = \theta$ is $y_{\gamma\theta}^{od}\ups_{od}$.
The equilibrium condition in Definition~\ref{def:bue} in this case is that, if $\P(\theta)>0$ and $y_{i\theta}^{od}>0$ for a path $i$ in $\Gamma_{od}$ when $\Theta = \theta$, then
\begin{equation}\label{eq:full_info}
	c_i (\theta,f^{\pi^{FI},y}(\theta)) \le
	c_j (\theta,f^{\pi^{FI},y}(\theta)) \quad \forall j \in \Gamma_{od}\,,
\end{equation}
namely, if some users $(o,d,u)$ receive message $\theta$ and travel on path $i$ in $\Gamma_{od}$, then path $i$ has to be no worse than any alternative path $j$ in $\Gamma_{od}$ when $\Theta = \theta$.

%We then show by a numerical example that in general they are both sub-optimal compared to the system optimum.
%In the rest of the section we shall assume that the delay functions are linear in the flow, that is
%$$
%\tau_e(\theta_e,f_e) = \frac{f_e}{\alpha_e} + \theta_e\,, \quad \forall e \in \mc E\,,
%$$
%with $\alpha_e>0$ deterministic and positive for every link. The next result proves that in this case, under the additional assumption that the full-information equilibrium is full support, the no-information policy and the full information policy achieve the same cost.

\subsection{Full information vs no information}
The next preliminary result states that, if the delay functions are affine and the BWEs corresponding to the no information and full information rule are full support, providing full information is equivalent to providing no information. The following example then shows that these rules do not achieve the optimal cost.
%\tcr{QUEST ASSUMPTION LA FACCIAMO PER TUTTO, MESSA QUI SEMBRA UN PO' NASCOSTA (di fatto sto assumendo che l'equilibrio no-info è full support e che l'equilibrio full-info è full support for every realization of $\theta$, mentre in seguito assumerò che il system-optimum è full support for every realization $\theta$ nel supporto, non sono esattamente la stessa cosa)}

\begin{assumption}\label{ass:linear}
	The delay functions are
	$$
	\tau_e(\theta_e,f_e) = \frac{f_e}{\alpha_e} + \theta_e\,, \quad \forall e \in \mc E\,,
	$$
	with $\alpha_e>0$ for every link $e$ in $\mc E$.
\end{assumption} 
\begin{proposition}\label{prp:no_full}
	Consider the no-information rule $\pi^{NI}$ defined in Section \ref{sec:no} and the full-information rule $\pi^{FI}$ defined in Section \ref{sec:full} and let Assumption \ref{ass:linear} hold true. Assume that $\pi^{NI}$ and $\pi^{FI}$ admit respectively a solution $y^{NI}$ and $y^{FI}$ of \eqref{eq:y} with full support, namely $y^{NI}$ such that $(y^{NI})^{od}_{\gamma m^*}>0$ for every two nodes $o,d$ in $\mc V$, $\gamma$ in $\Gamma_{od}$ and $y^{FI}$ such that $(y^{FI})^{od}_{\gamma\theta}>0$ for every two nodes $(o,d)$ in $\mc V$, $\gamma$ in $\Gamma_{od}$ and $\theta$ in $\text{supp}(\Theta)$. Then, $C(f^{\pi^{NI}}) = C(f^{\pi^{FI}})$, i.e., the two rules achieve the same cost.
\end{proposition}
\begin{proof}
See Appendix \ref{app:linear}.
\end{proof}\medskip

\begin{ex}\label{ex:1}
Consider a network with two nodes $\mc V = \{o,d\}$ and $|\mc E| = 2$ parallel links from $o$ to $d$, and let $\ups_{od} = 1$. Consider a network state $\Theta$ with probability distribution
$$
\Theta = \begin{cases} (2,9/5 + x)\,, \ \text{with probability} \ $1/2$\,, \\
	(2,9/5 - x)\,, \ \text{with probability} \ $1/2\,,$ 
\end{cases}
$$
where $x$ is a real number in $[0,4/5]$. Let the delay functions be 
$$\tau_1(\theta_1, f_1)=f_1+\theta_1\,, \quad \tau_2(\theta_2,f_2) = f_2 + \theta_2.$$ 
Using \eqref{eq:system_opt_full}, the system-optimum gets
%$$
%f^*(\theta) = \left(\frac 12 + \frac{\theta_2-\theta_1}4, \frac 12 + \frac{\theta_1-\theta_2}4\right)
%$$
%so that
$$\begin{aligned}
&f^*((2,9/5+x)) = \left(\frac 9{20} + \frac{x}4, \frac {11}{20} -\frac{x}4 \right), \\[3pt]
&f^*((2,9/5-x)) = \left(\frac 9{20} - \frac{x}4, \frac {11}{20} +\frac{x}4\right)\,,
\end{aligned}$$
and the corresponding cost parametrized by $x$ is 
$$C_x(f^*) = \frac{479}{200}-\frac{1}{8}x^2.$$
Using \eqref{eq:no_info_eq} the no information BWE gets
$$
f^{\pi^{NI}}(\theta) = \left(\frac 25, \frac 35\right)\,, \quad \forall \theta \in \text{supp}(\Theta)\,,
$$
and using \eqref{eq:full_info} the full information BWE gets
$$\begin{aligned}
&f^{\pi^{FI}}((2,9/5+x)) = \left(\frac 25 + \frac x2, \frac 35 - \frac x2\right)\,, \\[3pt] &
f^{\pi^{FI}}((2,9/5+x)) = \left(\frac 25 - \frac x2, \frac 35 + \frac x2\right)\,.
\end{aligned}$$
Observe that the full information BWE depends on the realization of the network state, while the no information BWE does not. Nonetheless, the two rules achieve the same cost
$$
C(f^{\pi^{NI}}) = C(f^{\pi^{FI}}) = \frac{12}5 \,,
$$
as established by Proposition \ref{prp:no_full} (in fact, in this example the two equilibria are full support). The cost of the two signaling rules proves to be independent of $x$ due to the linearity of the delay functions and their symmetry with respect to $x$. Both the signaling rules are suboptimal, as illustrated in Figure~\ref{fig:ex1}.
\begin{figure}
	\centering
	\includegraphics[width=0.8\linewidth]{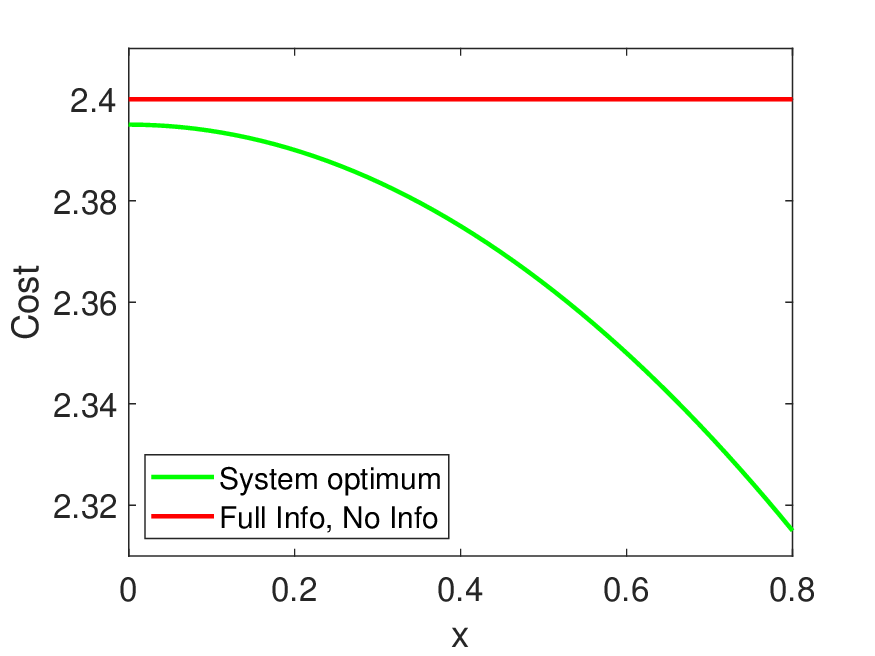}
	\caption{The cost of the system-optimum flow, the full information rule, and the no information rule for Example \ref{ex:1}.}
	\label{fig:ex1}
\end{figure}
\end{ex}
%$$
%\Theta_1 = 1,\quad \Theta_2 = \begin{cases}1/2 \ & \text{with probability} \ $1-p$ \\
%	3/2 & \text{with probability} \ $p\,,$ 
%\end{cases}
%$$
%with $p$ in $[0,1]$. Let the delay functions be 
%$$\tau_1(\theta_1, f_1)=2f_1+\theta_1\,, \quad \tau_2(\theta_2,f_2) = f_2 + \theta_2.$$
%Since $\Theta_1$ is deterministic, let $f^*(\theta_2)$ denote the system-optimum as a function of $\theta_2$. Direct computation yields
%$$
%f^*(\theta_2) = \left(\ba{c}(\theta_2 +1)/6 \\ (5-\theta_2)/6\ea\right)\,,
%$$
%and the corresponding cost as a function of $p$ is 
%$$C_p(f^*) = \frac{5}{3}+\frac{7}{12}p.$$
%
%Let $\pi^{FI}$ and $\pi^{NI}$ denote the full-information and the no information policy, where the index $\gamma$ is omitted the origin-destination pair is unique. By some computation, one can see that the full-information equilibrium and no-information equilibrium are
%$$
%f^{\pi^{FI}}(\theta_2) = \left(\ba{c}\theta_2 / 3 \\ 1 - \theta_2/3\ea\right), \quad f^{\pi^{NI}}(\theta_2) = \left(\ba{c}1/6 + p/3 \\ 5/6 - p/3\ea\right).
%$$
%The full-information equilibrium is computed by \eqref{eq:full_info} and depends on the realization of the network state, while the no-information equilibrium is computed by \eqref{eq:no_info_eq} and does not depend on $\theta$ by construction. Note also that $f^{\pi^{NI}}$ depends on $\E[\Theta]$, hence on $p$. The cost of the two policies result to be the same, i.e.,
%$$
%C(f^{\pi^{NI}}) = C(f^{\pi^{FI}}) = \frac 53 + \frac 23 p \,.
%$$

Example \ref{ex:1} shows that providing full information or no information is in general suboptimal. In fact, it is a well known fact that public information is inefficient in achieving optimality in routing games \cite{das2017reducing,tavafoghi2017informational,zhu2022information}. In the next section we shall investigate when optimality can be achieved by private signaling. Then, in Section \ref{sec:examples} we shall conclude Example \ref{ex:1} by showing that private signaling outperforms full information and no information provision.

\section{Main results}
\label{sec:results}
This section is divided in three parts. In the first subsection we establish the revelation principle, a general result that allows to restrict the attention to a special class of signaling rules. The second subsection answers our main question by providing sufficient and necessary conditions for optimality of the best signaling rule on certain network topologies. Finally, in the third subsection we provide more interpretable sufficient conditions for optimality.

\subsection{Revelation principle}
\label{sec:pb_ref}
Before establishing our results, we introduce the some important notions about signaling rules.
\begin{definition}[Direct and obedient rule]\label{def:ob}
A rule $\pi$ is called \emph{direct} if $\mathcal M = \Gamma$ and $$\pi_\gamma^{od}(\theta) = 0\,, \quad \forall o,d \in \mc V\,, \ \forall \gamma  \notin \Gamma_{od}\,, \ \forall \theta \in \mb \R_+^{\mc E}\,,
$$ 
A direct rule $\pi$ is called \emph{obedient} if every $y$ in $\R^{\mc V \times \mc V \times \Gamma \times \Gamma}_+$ such that \eqref{eq:con:y} and
\be\label{eq:obedient}
y^{od}_{ij} = \delta^{(i)}_j %\mathbf 1_{\mc P_{\gamma}}(i)
\,, \quad \forall o,d \in \mc V, \ \forall i, j \in \Gamma_{od}
\ee 
is a solution of \eqref{eq:y}. 
\end{definition}

In plain words, a rule is direct if messages correspond to path recommendations and if users $(o,d,u)$ can receive only path recommendations in $\Gamma_{od}$. A direct rule is obedient if no one has interest in deviating from the received recommendation.
\begin{proposition}\label{prp:obedience}
\emph{(i)} For an obedient rule $\pi$, the corresponding BWE is
\be\label{eq:flow_obedient}
\!\!\!f^\pi_e(\theta) = \sum_{o,d \in \mc V} \ups_{od}\sum_{\gamma \in \Gamma_{od}} A_{e\gamma} \pi_{\gamma}^{od}(\theta)\,,\ \forall e \in \mc E, \ \forall \theta \in \mathbf \R_+^{\mc E}\,.
\ee

\noindent \emph{(ii)} A direct rule $\pi$ is obedient if and only if
\be\label{eq:obedience}
\int_{\R_+^{\mc E}} (c_i(\theta, f^{\pi}(\theta))-c_j(\theta,f^{\pi}(\theta)) \pi_i^{od} (\theta) d\P(\theta) \leq 0
\ee
for every two nodes $o,d$ in $\mc V$ and two paths $i,j$ in $\Gamma_{od}$, with $f^{\pi}$ defined in \eqref{eq:flow_obedient}.
\end{proposition}
\begin{proof}
\emph{(i)} By definition of obedient rule, every $y$ in $\R^{\mc V \times \mc V \times \Gamma \times \Gamma}_+$ that satisfies \eqref{eq:con:y} and \eqref{eq:obedient} is a solution of \eqref{eq:y}. Therefore, by Definition \ref{def:bue}, the BWE $f^\pi$ has entries
\begin{align*}
f_e^{\pi}(\theta) & = \sum_{o,d \in \mc V} \sum_{i \in \Gamma_{od}} A_{ei} \sum_{m \in \mathcal M} y_{im}^{od} \pi_m^{od}(\theta) \ups_{od} \\[5pt]
& = \sum_{o,d \in \mc V} \sum_{i \in \Gamma_{od}} A_{ei} \sum_{j \in \Gamma} \delta_i^{(j)} \pi_j^{od}(\theta) \ups_{od} \\[5pt]
& = \sum_{o,d \in \mc V} \ups_{od} \sum_{i \in \Gamma_{od}} A_{ei} \pi_i^{od}(\theta)\,,
\end{align*}
where the second equivalence follows from the fact that obedient rules are direct by definition and from \eqref{eq:obedient}.

\emph{(ii)} 
%This follows from \eqref{eq:obedient}.
It follows from Definition \ref{def:bue} that $\pi$ must satisfy \eqref{eq:flow_equilibrium}
for every two nodes $o,d$ in $\mc V$, two paths $i,j$ in $\Gamma_{od}$ and message $m$ in $\mc M$ such that $y_{im}^{od}>0$, where $y$ is a solution of \eqref{eq:y}. 
Since $\pi$ is an obedient rule, $\mc M = \Gamma$, and every $y$ that satisfies \eqref{eq:con:y} and \eqref{eq:obedient} (i.e., such that for every $i$ in $\Gamma_{od}$, it holds $y_{ii}^{od} = 1$ and $y_{ij}^{od} = 0$) is a solution of \eqref{eq:y}. The statement follows from these two facts. 
%A rule $\pi$ is obedient if and only if $\mc M = \Gamma$ and if the BWE of $\pi$ can be obtained by \eqref{eq:flow_y} with $y$ that satisfies \eqref{eq:obedient}. The statement follows from the fact that $\mc M = \Gamma$ by definition of obedient policy and from the fact that for every $i$ in $\Gamma_{od}$, it holds $y_{ii}^{od} = 1$ and $y_{ij}^{od} = 0$
%Therefore, a rule is obedient if and only if \eqref{eq:obedience} holds true for every two nodes $o,d$ in $\mc V$ and for every two paths $i,j$ in $\Gamma_{od}$.
\end{proof}\medskip

In plain words, constraints \eqref{eq:obedience}, referred to as \emph{obedience constraints}, state that a rule $\pi$ is obedient if every path $i$ in $\Gamma_{od}$ that is recommended to some users $(o,d,u)$ is no worse than any alternative path $j$ in $\Gamma_{od}$ according to the posterior belief of users $(o,d,u)$ that receive the recommendation to travel across path $i$. The set of obedient rules is
$$
\Pi = \{\pi \in \Pi_{\Gamma}: \eqref{eq:obedience} \ \text{holds true} \ \forall o,d \in \mc V, \ \forall i,j \in \Gamma_{od}\}\,.
$$

The next fundamental result states that, for every (possibly not direct) rule $\pi$, there always exists a direct and obedient rule $\tilde \pi$ such that $f^{\pi} = f^{\tilde \pi}$. This in turn implies that the two rules achieve the same cost. Therefore, it is without loss of generality to restrict the attention to obedient rules. Such a result is known in the literature of finite strategic games as \emph{revelation principle} \cite{bergemann2019information,bergemann2016bayes}.
\begin{theorem} \label{rv}
Consider a finite message set $\mc M$, a rule $\pi$ in $\Pi_{\mc M}$ and let $y$ be a solution of \eqref{eq:y} for rule $\pi$. Then, the rule $\tilde \pi$ in $\Pi_\Gamma$ defined by
\be\label{eq:tilde_pi}
\tilde\pi_\gamma^{od}(\theta) = \sum_{m \in \mc M}y_{\gamma m}^{od}\pi_{m}^{od}(\theta)\,, \quad \forall o,d \in \mc V, \ \forall \gamma \in \Gamma\,,
\ee
is direct and obedient. Moreover, $f^{\pi} = f^{\tilde \pi}$ and therefore $\pi$ and $\tilde \pi$ achieve the same cost.
\end{theorem}
\begin{proof}
%\tcr{(riguardo all'unicità di $y$ per una certa policy, dovrebbe aver a che fare con l'unicità dell'equilibrio sui path. Infatti, se $\tilde \pi$ è una policy obedient, il suo equilibrio sarà semplicemente la proezione sugli archi; ma dato che l'equilibrio è unico, potranno esserci altre $y$ buone solo c'è altro modo di proiettare il flusso sui path sugli archi)}
%Consider a rule $\pi$ and let $y$ be a solution of \eqref{eq:y} for that rule. 
Consider the rule $\tilde\pi$ defined by \eqref{eq:tilde_pi} and note that for every two nodes $o,d$ in $\mc V$
\be\label{eq:tilde_pi1}
\ba{rcl}
\ds \sum_{\gamma \in \Gamma} \tilde \pi_\gamma^{od}(\theta) & = & \ds \sum_{\gamma \in \Gamma} \sum_{m \in \mc M} y_{\gamma m}^{od} \pi_m^{od}(\theta) \\[10pt]
& = & \ds \sum_{m \in \mc M} \pi_m^{od}(\theta) \sum_{\gamma \in \Gamma} y_{\gamma m}^{od} \\[10pt]
& = & \ds \sum_{m \in \mc M} \pi_m^{od}(\theta) \\[10pt]
& = & 1
\ea
\ee 
for every $\theta$ in $\R_+^{\mc E}$,
where the third equivalence follows from \eqref{eq:con:y}.
Moreover, since $y_{\gamma m}^{od} = 0$ for every $\gamma$ not in $\Gamma_{od}$ by \eqref{eq:con:y}, it follows that
\be\label{eq:tilde_pi2}
\tilde\pi_\gamma^{od}(\theta) = 0\,, \quad \forall \gamma \notin \Gamma_{od}\,, \ \forall \theta \in \R_+^{\mc E}\,. 
\ee
Together, conditions \eqref{eq:tilde_pi1}-\eqref{eq:tilde_pi2} and the fact that $\tilde\pi$ has non-negative entries by construction prove that $\tilde \pi$ is a direct rule. We now assume that $\tilde\pi$ is obedient, so that the corresponding BWE $f^{\tilde\pi}$ is given by \eqref{eq:flow_obedient} and prove that \eqref{eq:obedience} holds true for every two nodes $o,d$ in $\mc V$ and two paths $i,j$ in $\Gamma_{od}$, so that Proposition \ref{prp:obedience}(ii) guarantees that $\tilde\pi$ is indeed an obedient rule.
%We now prove that $\tilde \pi$ satisfies \eqref{eq:obedience} with $f^{\tilde \pi}$ given by \eqref{eq:flow_obedient}, so that $\tilde\pi$ is an obedient rule by Proposition \ref{prp:obedience}, and that $f^{\tilde\pi} = f^{\pi}$. To this end,
%To show that it is also obedient, 
%consider an arbitrary $\tilde y$ in $\R^{\mc V \times \mc V \times \Gamma \times \Gamma}_+$ that satisfies \eqref{eq:con:y} and such that
%\be\label{eq:tilde_y}
%\tilde y^{od}_{ij} = \delta^{(i)}_j\,, \quad \forall o,d \in \mc V, \ \forall i, j \in \Gamma_{od}\,.
%\ee
%We now prove that $f^{\tilde \pi, \tilde y} = f^{\pi,y}$ and that $f^{\tilde \pi, \tilde y}$ is the BWE for rule $\tilde \pi$. To this end, note that
First, note that
\be\label{eq:flow}
\ba{rcl}
f^{\tilde\pi}_e(\theta) & = & \ds \sum_{o,d \in \mc V} \ups_{od} \sum_{\gamma \in \Gamma_{od}} A_{e\gamma} \tilde\pi_\gamma^{od}(\theta) \\[13pt]
& = & \ds \sum_{o,d \in \mc V} \ups_{od} \sum_{\gamma \in \Gamma_{od}} A_{e\gamma} \sum_{m \in \mathcal M} y_{\gamma m}^{od} \pi_m^{od}(\theta)\\[13pt]
& = & f^{\pi}_e(\theta)\,,
%& = & \ds \sum_{o,d \in \mc V} \ups_{od} \sum_{i \in \Gamma_{od}} A_{ei} \sum_{j \in \Gamma} \tilde\pi_j^{od} (\theta) \tilde y_{ij}^{od} \\[13pt]
%& = & \ds \sum_{\gamma \in \Gamma} v_{\gamma} \sum_{j \in \mc P_\gamma} \sum_{i \in \mathcal P_\gamma} A_{ei} \tilde\pi_j^\gamma (\theta) \tilde y_{ij}^\gamma \\[13pt]
%& = & f_e^{\tilde \pi, \tilde y}(\theta)\,,
\ea 
\ee
%\be\label{eq:flow}
%\ba{rcl}
%f^{\pi}_e(\theta) & = & \ds \sum_{o,d \in \mc V} \ups_{od} \sum_{\gamma \in \Gamma_{od}} A_{e\gamma} \sum_{m \in \mathcal M} y_{\gamma m}^{od} \pi_m^{od}(\theta)\\[13pt]
%& = & \ds \sum_{o,d \in \mc V} \ups_{od} \sum_{\gamma \in \Gamma_{od}} A_{e\gamma} \tilde\pi_\gamma^{od}(\theta) = f^{\tilde\pi}_e(\theta)\\[13pt]
%%& = & \ds \sum_{o,d \in \mc V} \ups_{od} \sum_{i \in \Gamma_{od}} A_{ei} \sum_{j \in \Gamma} \tilde\pi_j^{od} (\theta) \tilde y_{ij}^{od} \\[13pt]
%%& = & \ds \sum_{\gamma \in \Gamma} v_{\gamma} \sum_{j \in \mc P_\gamma} \sum_{i \in \mathcal P_\gamma} A_{ei} \tilde\pi_j^\gamma (\theta) \tilde y_{ij}^\gamma \\[13pt]
%%& = & f_e^{\tilde \pi, \tilde y}(\theta)\,,
%\ea 
%\ee
for every link $e$ in $\mc E$ and $\theta$ in $\R_+^{\mc E}$, hence $f^{\tilde\pi} = f^\pi$.
%where the third equivalence follows from \eqref{eq:tilde_pi2}-\eqref{eq:tilde_y} and the last one from \eqref{eq:flow_y} and from the fact that $\tilde \pi$ is direct, namely, $\mc M = \Gamma$. 
Since $f^{\pi}$ is the BWE for rule $\pi$, by Definition \ref{def:bue}, for every two nodes $o,d$ in $\mc V$, path $i$ in $\Gamma_{od}$ and message $m$ in $\mc M$ such that $y_{im}^{od}>0$, it holds true that
$$
	\int_{\R_+^{\mc E}} (c_i(\theta, f^{\pi}(\theta))-c_j(\theta,f^{\pi}(\theta)) \pi_m^{od} (\theta) d\P(\theta) \leq 0
$$
for every path $j$ in $\Gamma_{od}$.
We now multiply this equation by $y_{im}^{od}$ and sum over $m$ in $\mathcal M$. By \eqref{eq:tilde_pi}, we get that for every two nodes $o,d$ in $\mc V$ and for every two paths $i,j$ in $\Gamma_{od}$,
$$
%\ba{rcl}
\ds \int_{\R_+^{\mc E}} (c_i(\theta, f^{\pi}(\theta))-c_j(\theta,f^{\pi}(\theta)) \tilde\pi_i^{od}(\theta) d\P(\theta) \le0\,.
%& = & \ds \int_{\R_+^{\mc E}} (c_i(\theta, f^{\tilde \pi,\tilde y}(\theta))-c_j(\theta,f^{\tilde \pi,\tilde y}(\theta)) \tilde\pi_i^{\gamma}(\theta) d\P(\theta)\,,
%\ea
$$
This proves that $\tilde\pi$ is an obedient rule.
The fact that $\pi$ and $\tilde \pi$ achieve the same cost follows from \eqref{eq:flow}.
%Therefore, if for aell an $i$ in $\mathcal P$ such that $\int_{\Theta}\tilde\pi_i(\theta)>0$ and for every $j$ in $\mathcal{P}$
%\begin{equation}
%\label{eq:bue_direct}
%\int_{\mathbf{\Theta}} c_i(\theta, A\tilde\pi(\theta))\tilde\pi_i(\theta)d\P(\theta)
%\leq 
%\int_{\mathbf{\Theta}} c_j(\theta, A\tilde\pi(\theta))\tilde\pi_i(\theta)d\P(\theta).
%\end{equation}
%Notice that \eqref{eq:bue_direct} implies that $f^{\tilde\pi}$ is also a Bayesian user equilibrium with $\tilde y=\mathbf{I}$ according to \ref{def:bue}
%Moreover, $f^{\tilde\pi}(\Theta)=A\tilde\pi(\Theta)=Ay\pi(\Theta)=f^{\pi}(\Theta)$, hence the costs for the two policies are the same, i.e. $C(\pi)=C(\tilde\pi)$.
\end{proof}\medskip

\begin{remark}
A version of the revelation principle for non-deterministic signal policies has been recently established in \cite{massicot2025strategic}.
\end{remark}

Our problem can be simplified in light of Theorem \ref{rv}. Indeed, if there exists a rule $\pi$ such that $f^\pi = f^*$, then Theorem \ref{rv} implies that there must exist a direct obedient rule $\tilde \pi$ such that $f^{\tilde\pi} = f^\pi = f^*$. Hence, we can restrict the attention to the set of obedient policies. This is a huge simplification, since it allows to assume that $\mc M = \Gamma$, thus eliminating the problem of designing the set of messages, which in principle might be every finite set. Moreover, the BWE for obedient policies is simply given by \eqref{eq:flow_obedient} and does not involves the solution of an inner optimization problem as established by Proposition \ref{prp:pot}. In the next section we shall observe that also the garbling step (i.e., in our language, considering non-deterministic signal policies) is unnecessary for certain network topologies.

\subsection{Optimality conditions}
\label{sec:opt_cond}
In this section we provide conditions under which the optimal rule $\pi^*$ achieves the optimal cost $C(f^*)$. This is equivalent to requiring that $f^{\pi^*} = f^*$. To this end, we shall work under Assumption~\ref{ass:linear}, which guarantees that the delay functions are linear, i.e.,
$$
\tau_e(\theta_e,f_e) = \frac{f_e}{\alpha_e} + \theta_e\,, \quad \forall e \in \mc E\,
$$
and under the following additional assumption.
\begin{assumption}\label{ass:A}
	\emph{(i)} There is a single pair of nodes $o,d$ in $\mc V$ such that $\ups_{od}>0$, so that it is without loss of generality to identify $\Gamma = \Gamma_{od}$ and neglect all paths with zero flow. Moreover, the link-path incidence matrix $A$ of the network is injective. 
\end{assumption} 

Note that under Assumption~\ref{ass:A} the throughput is scalar and it is without loss of generality to assume $\ups=1$.

\begin{remark}
	For undirected networks, Assumption \ref{ass:A} covers the class of linearly independent networks, defined as the networks with a single origin-destination pair where every path possesses at least one exclusive link \cite{milchtaich2006network}.
\end{remark}

\begin{remark}
An important consequence of Assumption \ref{ass:A} is the existence of a unique $z^*: \R_+^{\mc E} \to \mc Z_\ups$ such that $f^* = Az^*$, referred to as \emph{system-optimum path flow}. This fact, together with the revelation principle, implies that a rule $\pi$ such that $f^\pi = f^*$ exists if and only if $\pi = z^*$ is an obedient rule, whose BWE (cf. Proposition \ref{prp:obedience}(i)) is
$$
f^{z^*}_e(\theta) = \sum_{\gamma \in \Gamma} A_{e\gamma} z^*_{\gamma}(\theta) = f_e^*(\theta)\,,\ \forall e \in \mc E, \ \forall \theta \in \mathbf \R_+^{\mc E}\,.
$$ 
Even more importantly, since $z^*$ is the only path flow such that $f^* = Az^*$ and since a generalization of the revelation principle holds true also for non-deterministic signal policies (cf. \cite{massicot2025strategic}), the garbling step (i.e., considering non-deterministic signal policies) is unnecessary. In fact, the system-optimum flow can be the BWE corresponding to an obedient rule only if a fraction of agents $z^*_\gamma(\theta)$ receives the recommendation to travel over path $\gamma$ when $\Theta = \theta$, which is true for deterministic fair signal policies.
\end{remark}

We now introduce some notation. We define the random variable
$$\mathbf b:= A'\Theta\,,$$
which represents the free-flow travel time of all paths in the network. Moreover, we define a symmetric matrix $M$ in $\R^{\Gamma \times \Gamma}$ by
$$M:=(2A'[\alpha]^{-1}A)^{-1}\,.$$
\begin{remark}
Note that $M$ is well defined as a consequence of $A$ being injective. Indeed, since $A$ is injective, for every $x \neq \mathbf 0$, it holds $ A x \neq \mathbf0$. Hence, since $[\alpha]$ is positive definite, for every $x \neq \mathbf 0$,
$$
x'M^{-1}x = x'(2A'[\alpha]^{-1}A)x = 2(Ax)'[\alpha]^{-1} A x > 0\,.
$$
This proves that $M^{-1}$ is positive definite and $M = (M^{-1})^{-1}$ is well defined and positive definite. Therefore, it also holds $\mathbf 1' M \mathbf 1 > 0$.
\end{remark}
%\begin{remark}\label{remark:delta}
%	Note that $\Delta$ belongs to $\R_+^{\mathcal P \times \mathcal P}$, but it has only $|\mc P|-1$ independent entries. Indeed, by construction it satisfies $\Delta_{kp} = \Delta_{ki} - \Delta_{pi}$ for every triple of paths $i,j,p$. Hence, the entire matrix can be deduced from a single row, and the diagonal is zero, as $\Delta$ is skew-symmetric matrix construction.
%\end{remark}

The next assumption states that the system optimum path flow $z^*$ is full-support for every realization of the network state. We can then establish our main result, which provides sufficient and necessary conditions under which optimality can be achieved by private signaling.
\begin{assumption}\label{ass:supp}
The system-optimum path flow $z^*(\theta)$ is full support for every $\theta$ in $\text{supp}(\Theta)$.
\end{assumption}

\begin{theorem}\label{thm:opt_gen}
	Let Assumptions \ref{ass:linear}-\ref{ass:supp} hold true. Then, $\pi^*=z^*$ and $C(f^{\pi^*}) = C(f^*)$ if and only if
	\begin{equation}\label{eq:ob_statement}
	\E\Big[(\mb b_i - \mb b_j)\big(W_i + \sum_{\gamma \in \Gamma}V_{i\gamma}(\mb b_\gamma - \mb b_i)\big)\Big] \le 0
	\end{equation}
	for every two paths $i,j$ in $\Gamma$,
	where
	\begin{equation}\label{eq:ob_gen}
	W_i = \frac{\sum_j M_{ij}}{\mathbf1'M\mathbf1}\,, \ \ V_{i\gamma} = \frac{\sum_{j,k}(M_{j\gamma} M_{ik} - M_{jk}M_{i\gamma})}{\mathbf1'M\mathbf1}\,,
	\end{equation}
	for every two paths $i,\gamma$ in $\Gamma$.
\end{theorem} 

\begin{proof}
We aim to prove that $\pi = z^*$ is an obedient rule. If so, 
\be\label{eq:fz}
f^{z^*} = Az^* = f^*\,,
\ee
where the first equivalence follows from Proposition \ref{prp:obedience}(i). Hence, $\pi = z^*$ is an optimal rule and $C(f^{z^*}) = f^*$, thus achieving optimality. It follows from Proposition \ref{prp:obedience}(ii) and from \eqref{eq:fz} that $\pi = z^*$ is an obedient rule if and only if
\begin{equation}\label{eq:ob2}
\int_{\R_+^{\mc E}} (c_i(\theta, Az^*(\theta))-c_j(\theta,Az^*(\theta)) (z^*)_i^{od} (\theta) d\P(\theta) \leq 0
\end{equation}
for every $i,j$ in $\Gamma$,
namely, if and only if, for every path $i$, the users that receive message $i$ do not have incentive in deviating to any alternative path $j$. 
%Let $\tilde c(\theta,z) = c(\theta,Az)$ be the vector of path costs as a function of the path flow $z$ and of the network state realization. 
%%Let $b=A'\theta$ denote the realization of the random variable $\mb b$.
Assumption \ref{ass:supp} and the KKT optimality conditions imply that, given $\theta$ in $\text{supp}(\Theta)$, $z^*(\theta)$ is the unique vector in $\R_+^\Gamma$ such that $\1'z^*(\theta) = 1$ that admits a positive number $\lambda(\theta)$ such that 
$$
	\sum_{e \in \mc E} \frac{2 A_{e\gamma}}{\alpha_e} (Az^*(\theta))_e + (A'\theta)_\gamma = \lambda(\theta)\,, \quad \forall \gamma \in \Gamma\,.
$$
More compactly, the pair $(z^*(\theta),\lambda(\theta))$ satisfies
\be\label{eq:lambda}
M^{-1}z^*(\theta) + A'\theta = \lambda(\theta) \mathbf1\,.
\ee
Now, for every $\theta$ in $\text{supp}(\Theta)$, define the vector $z^W(\theta)$ (possibly with negative entries) that satisfies
$$
\sum_{e \in \mc E} \frac{A_{e\gamma}}{\alpha_e} (Az^W(\theta))_e + (A'\theta)_\gamma = \mu(\theta)\,, \quad \forall \gamma \in \Gamma\,,
$$
for a positive number $\mu(\theta)$ such that $\mathbf1' z^W(\theta) = 1$ holds true. In matrix form, this relation reads
\begin{equation}\label{eq:zw}
\frac 12 M^{-1} z^W(\theta) +  A'\theta = \mu(\theta) \mathbf1\,,
\end{equation}
%Notice that $z^W(\theta)$ corresponds to the user equilibrium when it is full-support, whereas $z^W(b)$ has non-positive entries when the equilibrium is not full-support.
%(actually, the dependence is on $\beta$). 
which guarantees uniqueness of $z^W$ since $M^{-1}$ is invertible.
Hence, for every $\theta$ in $\text{supp}(\Theta)$, there exist two positive numbers $\mu(\theta),\lambda(\theta)$ such that
\begin{equation*}\label{eq:opt}
z^*(\theta) = M(\lambda(\theta) \mathbf1 - A'\theta)\,, \quad
z^W(\theta) = 2M(\mu(\theta) \mathbf1 - A'\theta).
\end{equation*}
Using these two equations,
\begin{equation}\label{eq:opt_w}
z^*(\theta) = z^W(\theta) + (\lambda(\theta)-2\mu(\theta))M\mathbf1 + MA'\theta.
\end{equation}
From now on, we omit for convenience of notation the dependence on $\theta$ when unnecessary. We now compute explicitly the terms inside the integral in \eqref{eq:ob2}. It follows from \eqref{eq:cost_path} and from Assumption \ref{ass:linear} that
\begin{equation}\label{eq:cost}
c(\theta,Az) = \frac12 M^{-1}z + A'\theta\,.
\end{equation} Hence,
\begin{equation}\label{eq:cost_dif}
\!\!\!\begin{aligned}
	& \ c_i(\theta,Az^*)- c_j(\theta,Az^*) \\[6pt]
	= & \  (M^{-1}z^*)_i/2 + (A'\theta)_i - (M^{-1}z^*)_j/2 - (A'\theta)_j \\[6pt]
	= & \ (M^{-1}z^W)_i/2 + (A'\theta)_i - (M^{-1}z^W)_j/2 - (A'\theta)_j \ +\\[6pt]
	+ & \ \ds \left(\lambda-2\mu\right)(M^{-1}M\mathbf1)_i - \left(\lambda-2\mu\right)(M^{-1}M\mathbf1)_j \ +\\[6pt]
	+ & \ \ds \frac{1}{2}(M^{-1}MA'\theta)_i - \frac{1}{2}(M^{-1}MA'\theta)_j \\[6pt]
	= & \ \frac{(A'\theta)_i-(A'\theta)_j}2\,,
\end{aligned}
\end{equation}
where the first equality follows from \eqref{eq:cost}, the second one from \eqref{eq:opt_w} and the last one from \eqref{eq:zw}. It follows from \eqref{eq:lambda} and from $\mathbf1'z^* = 1$ that
\begin{equation*}\label{eq:lambda_z}
\lambda(\theta) = \frac{1+\mathbf1'MA'\theta}{\mathbf1'M\mathbf1}\,, \ z^*(\theta) = \frac{1+\mathbf1'MA'\theta}{\mathbf1'M\mathbf1}M \mathbf1 - MA'\theta.
\end{equation*}
Component-wise, by letting for simplicity $A'\theta = b$,
\be\label{eq:z}
\ba{rl}
\! & \!\!\! z^*_i\\[6pt]
\! = & \!\!\! \ds \frac{\ds \sum_j M_{ij}}{\ds \mathbf1'M\mathbf1} + \frac{\ds \sum_{j,k,\gamma} M_{jk} M_{i\gamma} b_k - \sum_{j,k,\gamma} M_{jk} M_{i\gamma}b_\gamma}{\ds \mathbf1'M\mathbf1} \\[10pt]
\! = & \!\!\! \ds \frac{\ds \sum_j M_{ij}}{\ds \mathbf1'M\mathbf1} + \frac{\ds \sum_{j,k,\gamma} \big(M_{jk}M_{i\gamma} (b_k-b_i) - M_{jk}M_{i\gamma} (b_\gamma-b_i)\big)}{\ds \mathbf1'M\mathbf1} \\[10pt]
\! = & \!\!\! \ds \frac{\ds \sum_j M_{ij}}{\ds \mathbf1'M\mathbf1} + \frac{\ds \sum_{j,k,\gamma} (M_{j\gamma}M_{ik}-M_{jk}M_{i\gamma}) (b_\gamma-b_i)}{\ds \mathbf1'M\mathbf1} \\[10pt]
\! = & \!\!\! \ds W_i + \sum_{\gamma \in \Gamma} V_{i\gamma} (b_\gamma-b_i).
\ea
\ee
Plugging this expression and \eqref{eq:cost_dif} inside \eqref{eq:ob2} we get \eqref{eq:ob_statement}, concluding the proof.
\end{proof}\medskip

%The proof of Theorem \ref{thm:opt_gen} relies on the revelation principle (cf. Theorem \ref{rv}), since $f^* = Az^*$ may be the equilibrium BWE of a rule $\pi$ if and only if the policy $\pi = z^*$ is obedient. Hence, the proof consists in verifying when $\pi = z^*$ is an obedient policy.

%\begin{remark}
%The conditions established in Theorem \ref{thm:opt_gen} are sufficient and necessary for optimality even in the more general setting whereby the planner is allowed to design a rule that randomizes among a set of obedient signaling rules for every network state realization, as considered in \cite{koessler2024correlated,massicot2022competitive}. This is due to the fact that, under our assumptions, the system optimum $z^*$ is unique.
%\end{remark}

Note that users that receive different messages may experience different travel times. This may be erroneously interpreted as a source of unfairness introduced by the planner. However, according to the definition of fair signal policy, all users with same origin and destination have the same probability of receiving each message, so that fairness is achieved in expectation. 
%This is a general fact that holds true in deterministic routing games at the system-optimum flow. \tcr{DETTO MALE} In contrast, the no information BWE is characterized by the fact that all routes have the same expected travel time.

Theorem \ref{thm:opt_gen} provides necessary and sufficient conditions for optimality. The drawback of the result is that these conditions are hard to interpret. In the next section we shall provide more intuitive sufficient conditions for optimality. In the rest of this section we shall prove that the optimality conditions provided in \cite{cianfanelli2023information,ambrogio2024information} for networks with parallel links are indeed implied by the general results of this paper. To this end, let
\begin{equation*}
	\begin{aligned}
		&\Delta_{ij} = \mathbb E[\mb b_i-\mb b_j]\,, \quad \forall i,j \in \Gamma\,, \\[5pt]
		&K_{el} = \mathbb E[\Theta_e\Theta_l]-\mathbb E[\Theta_e]\mathbb E[\Theta_l]\,, \quad \forall e,l \in \mathcal E\,,
	\end{aligned}
\end{equation*}
namely, $\Delta_{ij}$ is the expected difference between the free-flow travel time of paths $i$ and $j$ and $K$ is the covariance matrix of $\mb b$.

\begin{corollary}\label{cor:n}
	Consider a network with $\mc V = \{o,d\}$ nodes and $|\mc E| = N$ parallel links from $o$ to $d$, and let Assumptions \ref{ass:linear} and \ref{ass:supp} hold true.
	Then, $\pi^* = f^*$ if and only if
	\be\label{eq:optN_cor}
	\! \sum_{\gamma \neq i} \alpha_p (K_{ii}+K_{\gamma j}-K_{\gamma i}-K_{ij}) \ge	\Delta_{ij}(2+\sum_{\gamma \neq i} \alpha_\gamma \Delta_{\gamma i})
	\ee
	for every $i,j$ in $\mc E$.
\end{corollary}

\begin{proof}
	Note that networks with parallel links satisfies Assumption \ref{ass:A}, hence Theorem \ref{thm:opt_gen} applies. Moreover, $\Gamma = \mc E$ and $M=[\alpha]/2$. Hence, from \eqref{eq:ob_gen},
	\begin{equation}\label{eq:WV}
		W_i = \frac{\alpha_i}{\sum_k \alpha_k}, \
		V_{i\gamma} = \frac{\alpha_\gamma\alpha_i}{\sum_k \alpha_k}\,, \quad \forall i, \gamma \in \mc E, i \neq \gamma\,.
	\end{equation}
	Plugging this expression into \eqref{eq:ob_statement}, we get that the necessary and sufficient condition for optimality is
	\begin{equation*}
		\E\bigg[(\mb b_i - \mb b_j)\bigg(\frac{\alpha_i}{\sum_k \alpha_k} + \sum_{\gamma \in \mc E}\frac{\alpha_\gamma\alpha_i}{2\sum_k \alpha_k}(\mb b_\gamma - \mb b_i)\bigg)\bigg] \le 0\,,
	\end{equation*}
	for every two links $i,j$ in $\mc E$, which is equivalent to
	\begin{equation*}
		2\Delta_{ij}+ \sum_{\gamma \in \mc E}\alpha_\gamma \E[(\mb b_\gamma - \mb b_i)(\mb b_i - \mb b_j)] \le 0\,.
	\end{equation*}
	Note that for networks with parallel links $\Theta = \mb b$. Hence, this is equivalent to
	\begin{equation*}
		2\Delta_{ij}+ \sum_{\gamma \neq i}\alpha_\gamma \left[\Delta_{ij}\Delta_{\gamma i} +K_{i\gamma}-K_{ii}-K_{\gamma j}+K_{ij}\right] \le 0\,,
	\end{equation*}
	which is equivalent to \eqref{eq:optN_cor}. This proves the statement.
\end{proof}\medskip

\begin{corollary}\label{cor:2}
Consider a network with $\mc V = \{o,d\}$ and $|\mc E| = 2$ parallel links from $o$ to $d$. Let Assumptions \ref{ass:linear} and \ref{ass:supp} hold true.
Then, $\pi^* = f^*$ if and only if
\be\label{eq:opt_2}
\sigma^2 \ge \max\{\Delta_{12}(2/\alpha_2 - \Delta_{12}),\Delta_{21}(2/\alpha_1 - \Delta_{21})\}\,,
\ee
where $\sigma^2 = \text{Var}(\Theta_1 - \Theta_2)$.
\end{corollary}\smallskip
\begin{proof}
Corollary \ref{cor:n} with $N=2$ links states that a sufficient and necessary condition for optimality is that
$$
\ba{rcl}
\Delta_{12}(2+\alpha_2\Delta_{21}) & \le & \alpha_2 (K_{11}+K_{22}-K_{12}-K_{21})\\[4pt]
\Delta_{21}(2+\alpha_1\Delta_{12}) & \le & \alpha_1(K_{11}+K_{22}-K_{12}-K_{21})\,.
\ea
$$
The proof follows from $\sigma^2 = K_{11} + K_{22} - 2K_{12}$, and from $\Delta_{21} = - \Delta_{12}$.
\end{proof}\medskip

In the next section we shall establish more intuitive sufficient conditions for optimality that follows from the sufficient and necessary conditions established by these results.

\subsection{Sufficient conditions for optimality}
This section is devoted to provide intuitive conditions for optimality based on the results of the previous section.

\begin{corollary}\label{cor:2bis}
	Consider a network with node set $\mc V = \{o,d\}$ and $|\mc E| = 2$ parallel links from $o$ to $d$. Let Assumptions \ref{ass:linear} and \ref{ass:supp} hold true and let $\Delta = \0$. 
	Then, $\pi^*=f^*$ and $C(f^{\pi^*}) = C(f^{*})$. 
\end{corollary}
\begin{proof}
	The proof follows immediately from Corollary \ref{cor:2} by noticing that \eqref{eq:opt_2} is satisfied when $\Delta = \0$, since $\sigma^2$ is non-negative by construction.
\end{proof}\medskip

\begin{corollary}\label{cor:nbis}
Consider a network with $\mc V = \{o,d\}$ and $|\mc E| = N$ parallel links from $o$ to $d$. Let Assumptions \ref{ass:linear} and \ref{ass:supp} hold true. Let:
\begin{enumerate}
	\item[(i)] $\Delta = \0$;
	\item[(ii)] $K$ be diagonal.
\end{enumerate}
Then, $\pi^* = f^*$ and $C(f^{\pi^*}) = C(f^*)$.
\end{corollary}
\begin{proof}
	The proof follows from Corollary \ref{cor:n} by noticing that the right-hand side of \eqref{eq:optN_cor} is zero because $\Delta = \0$, and the left-hand side is non-negative because $K$ is a diagonal matrix with non-negative by definition of covariance matrix.
\end{proof}\medskip

\begin{proposition}\label{prp:suff_general} 
Let Assumptions \ref{ass:linear}-\ref{ass:supp} hold true. Let:
\begin{enumerate}
\item[(i)] $\Delta = \0$;
\item[(ii)] $K$ be diagonal;
\item[(iii)] $V$ be a Metzler matrix, i.e., $V_{ij} \ge 0$ for every $i \neq j$.
\end{enumerate}
Then, $\pi^*=z^*$ and $C(f^{\pi^*}) = C(f^*)$.
\end{proposition}

\begin{proof}
The proof is contained in Appendix \ref{app:metzler}.
\end{proof}\medskip

\begin{remark}\label{rem:Metzler}
Note that, as a consequence of \eqref{eq:z},
%\begin{remark}
%The matrix $V$ has some properties. Let $v = M\mathbf1$ and $V = vv' - (\mathbf1'M\mathbf1) M$.
%Notice that
%$$V \mathbf1 = (M \mathbf1 \mathbf1' M) \mathbf1 - (\mathbf1' M \mathbf1) M \mathbf1 = \mathbf0.$$
%Moreover, for most of the networks and assignments of $\alpha$, $V$ is a Metzler matrix. Indeed, it is possible to show that (as it is done in \eqref{eq:zi} in the proof of Theorem \ref{opt_gennet})
%\begin{equation*}\label{eq:int_theta}
%V_{ip} = \frac{\partial z_i^*}{\partial b_p}\,,
%\end{equation*}
%where $b=A'\theta$.
%This means that 
requiring that $V$ is Metzler is equivalent to requiring that the optimal flow on a path is non-decreasing in the free-flow delay of all alternative paths. This condition is trivially satisfied when the network has parallel links, as \eqref{eq:WV} shows, but does not hold true in general, as illustrated in the next section by an example.
\end{remark}

The three results established in this section provide sufficient conditions for optimality. These sufficient conditions become more restrictive as the network topology becomes more general. In particular, if the network has two parallel links, Corollary \ref{cor:2bis} states that optimality is guaranteed if the links have the same expected free-flow travel time. If the network has $N$ parallel links, Corollary \ref{cor:nbis} states that optimality is guaranteed if, additionally, the free-flow delay of all links are uncorrelated. Proposition \ref{prp:suff_general} states that, for arbitrary networks with injective link-path incidence matrix, optimality is guaranteed if the links have equal expected free-flow delay, the entries of $\Theta$ are uncorrelated and $V$ is a Metzler matrix. 
%\end{remark}

\section{Examples}\label{sec:examples}

This section is dedicated to the examples and has two purposes. First, the examples show that, if a sufficient condition in Corollary \ref{cor:nbis} and Proposition \ref{prp:suff_general} is not met, optimality might not be achievable by private signaling. Moreover, the examples help illustrating why these sufficient conditions are needed.
\addtocounter{ex}{-1}
\begin{ex}[continued]
Figure \ref{fig:cont} illustrates the cost of the optimal signaling rule $\pi^*$ as a function of $x$, in comparison with the optimal cost and the cost of the no information and full information rules. Note that, when $x=0$, the optimal rule achieves same performance of the no information and full information rules. This is due the fact that, when $x=0$, the network state is deterministic, hence the BWE coincides with the classical notion of Wardrop equilibrium for every signaling rule \cite{koessler2024correlated}. When $x \ge 3/5$, the optimal rule $\pi^*$ achieves the optimal cost. This is consistent with Corollary \ref{cor:2}, which states that optimality is achieved if and only if
$$x^2 = \sigma^2 \ge \frac 9{25}.$$
In fact, this example satisfies all the assumptions of Corollary~\ref{cor:2}.
When $x$ belongs to $(0,3/5)$, the optimal rule $\pi^*$ outperforms the no information and the full information rule and is suboptimal compared to the system-optimum flow. This regime is characterized by the fact that rule $\pi = f^*$ is not an obedient rule. Observe that, in absence of information or revealing full information the system cost is independent of $x$ due to the fact that in this example $\E[\Theta]$ does not depend on $x$. In contrast, when the planner discloses information privately according to the optimal rule $\pi^*$, the system cost is monotonically decreasing in $x$. This shows that the uncertainty is beneficial for the system, and can be exploited by an informed planner as a mechanism to reduce the network congestion.
\begin{figure}
	\centering
	\includegraphics[width=0.8 \linewidth]{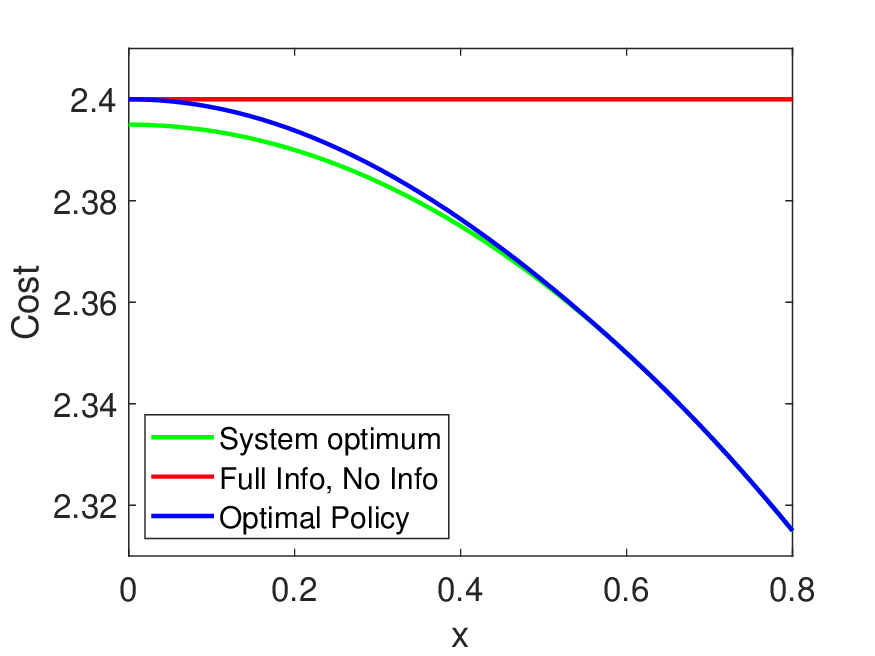}
	\caption{The cost of the system-optimum flow, the full information rule, the no information rule and the optimal rule in Example \ref{ex:1}.}
	\label{fig:cont}
\end{figure}
\end{ex}

\begin{ex}\label{ex:2vs3}
Consider a network with $\mc V=\{o,d\}$ and $|\mc E| = 3$ links from $o$ to $d$ with affine delay functions as per Assumption \ref{ass:linear}. Let $\alpha_1=\alpha_2=\alpha_3=1$, and assume that 
\be\label{eq:theta_ex2}
\Theta_1(\mb h) = c + \mb h, \quad \Theta_2(\mb h) = c + w \mb h, \quad \Theta_3(\mb h) = c\,, 
\ee
where $c,w>0$ are two positive numbers and $\mb h$ is a random variable with $\E[\mb h] = 0$ and $\E[\mb h^2] = \sigma^2 > 0$.
Observe that
\be\label{eq:DK_ex}
\Delta = \0, \quad 
K = \left(\begin{matrix} \sigma^2 & w \sigma^2 & 0 \\
w \sigma^2 & w^2 \sigma^2 & 0 \\
0 & 0 & 0\end{matrix}\right)\,.
\ee
While Corollary \ref{cor:2bis} states that $\Delta = \mb 0$ is a sufficient condition for optimality for networks with $2$ parallel links, this example shows that such condition is not sufficient when the network has $3$ or more parallel links. In particular, plugging \eqref{eq:DK_ex} into the obedience constraint \eqref{eq:optN_cor} with $i=1, j=2$ gets satisfied if and only if $w \ge 2$ or $w \le 1$. In plain words, users that receive message $\gamma_1$ under rule $\pi = f^*$ prefer path $\gamma_1$ over path $\gamma_2$ if and only if $w \ge 2$ or $w \le 1$, whereas if $w$ belongs to $(1,2)$ users that receive message $\gamma_1$ prefer to deviate to path $\gamma_2$ and optimality cannot be achieved. This fact has the following explanation. Assume that $w < 1$. Then,
%and let $h$ be the realization of $\mb h$, which in turn determines by \eqref{eq:theta_ex2} the network state realization $\theta(h)$ and the system-optimum flow $f^*(h)$. 
direct computation yields
$$
\mb h \ge 0 \implies \begin{cases}
\Theta_3(\mb h) \le \Theta_2(\mb h) \le \Theta_1(\mb h) \\
 f_3^*(\Theta(\mb h)) \ge f_2^*(\Theta(\mb h)) \ge f_1^*(\Theta(\mb h))
\end{cases}$$ 
$$
\mb h < 0 \implies \begin{cases}
\Theta_3(\mb h) > \Theta_2(\mb h) > \Theta_1(\mb h) \\ f_3^*(\Theta(\mb h)) < f_2^*(\Theta(\mb h)) < f_1^*(\Theta(\mb h)).
\end{cases}
$$ 
Let $\E_\gamma[\cdot]$ be the expected value of a random variable according to the posterior belief of users that received message $\gamma$. The equations above, together with Proposition \ref{prp:bayes}(ii), imply that the posterior belief of users that receive message $\gamma_1$ under rule $\pi = f^*$ is that $\E_1[\mb h] < 0$. Hence, in particular, $\E_1[\Theta_1(\mb h)] < \E_1[\Theta_2(\mb h)] < \E_1[\Theta_3(\mb h)]$, which make users that receive message $\gamma_1$ prefer path $\gamma_1$ over paths $\gamma_2$ and $\gamma_3$.
Instead, if $w > 1$, it holds
$$
\mb h >0 \implies \begin{cases}
\Theta_3(\mb h) < \Theta_1(\mb h) < \Theta_2(\mb h) \\ 
f_3^*(\Theta(\mb h)) > f_1^*(\Theta(\mb h)) > f_2^*(\Theta(\mb h))
\end{cases}
$$
$$
\mb h \le 0 \implies 
\begin{cases}
\Theta_3(\mb h) \ge \Theta_1(\mb h) \ge \Theta_2(\mb h) \\ f_3^*(\Theta(\mb h)) \le f_1^*(\Theta (\mb h)) \le f_2^*(\Theta(\mb h))\,.
\end{cases}
$$
From this inequality it seems that when $w>1$ the sign of $\E_1[\mb h]$ is unclear under rule $\pi = f^*$, so that users that receive message $\gamma_1$ are not able to order the links based on their posterior expected free-flow delay. However, by direct computation, the system optimum flow is
$$
f_1^*(\Theta(\mb h)) = \frac{2+(w-2)\mb h}{6}.
$$
Hence, when $w > 2$, $f_1^*(\Theta(\mb h))$ is increasing in $\mb h$, which implies by Proposition \ref{prp:bayes}(ii) that $\E_1[\mb h] > 0$. Since $w \ge 2$, this in turn implies by \eqref{eq:theta_ex2} that $\E_1[\Theta_1(\mb h)] \le \E_1[\Theta_2(\mb h)]$, hence users that receive message $\gamma_1$ prefer path $\gamma_1$ over path $\gamma_2$ and the obedience constraint \eqref{eq:optN_cor} with $i=1$ and $j=2$ is satisfied. In contrast, when $w$ belongs to $(1,2)$, similar considerations prove that $\E_1[\Theta_1(\mb h)] > \E_1[\Theta_2(\mb h)]$, which make users that receive message $\gamma_1$ deviate to path $\gamma_2$.

%I expect that in this case users prefer link $3$, i.e., $OB_{13}$ is violated. Indeed, $OB_{13}$ reads
%$$
%-2+w \le 0
%$$
%which is violated when $w \ge 2$, as expected. On the other hand, when $w < 2$, $f_1^*(\eta)$ is decreasing in $\eta$, suggesting that $\E_1[\eta] < 0$. Therefore, $\E_1[b_1] > \E_1[b_2]$, hence users of type $1$ prefer link $2$ over link $1$ and $OB_{12}$ is not satisfied.
\end{ex}

The previous example has shown that when the network has parallel links, the correlations of the network state entries may prevent optimality to be achieved. The next example shows that, for networks with non-parallel links and with injective link-path incidence matrix $A$, then the expected free-flow delay of all paths being equal and the entries of network state being uncorrelated may not be sufficient if $V$ is not Metzler.
\begin{ex}\label{ex3}
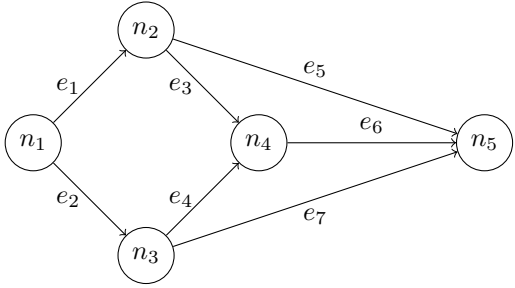
\begin{figure}
	\centering
	\begin{tikzpicture}
		\node [circle,draw] (1) at (0,0) {$n_1$};
		\node [circle,draw] (2) at (1.5,1.5)
		{$n_2$}; 
		\node [circle,draw] (3) at (1.5,-1.5) {$n_3$};
		\node [circle,draw] (4) at (3,0)
		{$n_4$}; 
		\node [circle,draw] (5) at (6,0)
		{$n_5$}; 
		
		\path (1) edge [->]
		node [left] {$e_{1}$} (2);
		\path (1) edge [->]
		node [left] {$e_{2}$} (3);
		\path (2) edge [->]
		node [left] {$e_{3}$} (4);
		\path (3) edge [->]
		node [left] {$e_{4}$} (4);
		\path (2) edge [->]
		node [above] {$e_{5}$} (5);
		\path (4) edge [->]
		node [above] {$e_{6}$} (5);
		\path (3) edge [->]
		node [below] {$e_{7}$} (5);
	\end{tikzpicture}
\caption{The network of Example \ref{ex3}. \label{fig:ex}}
\end{figure}

Consider the network in Figure \ref{fig:ex}, with
\begin{equation*}
\alpha = \bigg(\frac{1}{2},1,10,1,1,\frac{1}{2},1\bigg)\,.
\end{equation*}
Assume that $\Theta_e$ is deterministic for every $e$ in $\mc E \setminus \{e_4\}$, so that $K$ is diagonal with
\be\label{eq:K_example}
K_{44} = 0, \quad K_{ee} = 0 \ \ \forall e \in \mc E \setminus \{e_4\}\,,
\ee 
and let
\begin{equation}\label{eq:exp_b}
\ba{c}
\E[\Theta_2] = \E[\Theta_5] = \E[\Theta_6] = 1, \\[6pt]
\E[\Theta_1] = \E[\Theta_7] = \E[\Theta_4] = 4, \\[6pt]
\E[\Theta_3] = 0.
\ea
\end{equation}
%\be b_1 = \begin{cases}3.9 & \text{with prob. } 0.5 \\
%4.1 & \text{with prob. } 0.5\,, \qquad
%\end{cases} 
%b_4 = \begin{cases}2.95 & \text{with prob. } 0.5 \\
%	3.05 & \text{with prob. } 0.5\,.
%\end{cases} 
%\ee
Label the paths as
\be\label{eq:path_ex}
\begin{aligned}
\gamma_1 = (e_1,e_5),  \quad \gamma_2 = (e_1,e_3,e_6)\,, \\ 
\gamma_3 = (e_2,e_4,e_6),  \quad \gamma_4 = (e_2,e_7)\,,
\end{aligned}
\ee
and note that
\begin{equation}\label{eq:beta_i_ex}
\E[\mb b_\gamma] = 5\,, \quad \forall \gamma \in \Gamma\,,
\end{equation}
hence $\Delta = \mb 0$.
%\be
%K = \left(\ba{ccccccc}
%4 & 0 & 0 & 0 & 0 & 0 & 0 \\
%0 & 0 & 0 & 0 & 0 & 0 & 0 \\
%0 & 0 & 0 & 0 & 0 & 0 & 0 \\
%0 & 0 & 0 & 1 & 0 & 0 & 0 \\
%0 & 0 & 0 & 0 & 0 & 0 & 0 \\
%0 & 0 & 0 & 0 & 0 & 0 & 0 \\
%0 & 0 & 0 & 0 & 0 & 0 & 0 \\
%\ea\right)\,.
%\ee
By numerical computation,
\begin{equation*}\label{eq:theta}
V \simeq \left(\begin{matrix}
-0.1129 & 0.0953 & -0.0432 & 0.0609 \\
0.0953 & -0.1393 & 0.0806 & -0.0367 \\
-0.0432 & 0.0806 & -0.1063 & 0.0689 \\
0.0609 & -0.0367 & 0.0689 & -0.0931
\end{matrix}\right)\,.
\end{equation*} 
Observe that $V$ is not Metzler, since $V_{13} = V_{31} < 0$, so that conditions (i)-(ii) of Proposition \ref{prp:suff_general} are satisfied and condition (iii) is not.
%We now aim to find a distribution of $b$ with diagonal $K$ such that $z^*$ satisfies Assumption \ref{ass} and is no longer an obedient policy, so that $\theta$ being not Metzler would be the only assumption of Proposition \ref{prp:no_corr_arb} that is violated. In particular, for agents that receive message $3$, path $1$ is better that path $3$, namely, $OB_{31}$ is violated. To prove this, observe that
%\be\label{eq:K}
%K_{11} = \frac{1}{100}\,, \quad K_{44} = \frac{1}{400}\,, \quad K_{ee} = 0 \ \ \ \forall e \neq 1,4\,.
%\ee
The constraint \eqref{eq:ob_statement} with $i=1$, $j=3$ reads
%$$
%\!\!\! 
%\begin{array}{rcl}
%	& & \E\left[\left(\mb b_3 - \mb b_1\right)\left(W_3 + \sum_{p}V_{3p}(\mb b_p - \mb b_3)\right)\right] \\[6pt]
%	& = & \E[V_{31}(\mb b_3-\mb b_1)(\mb b_1-\mb b_3)]+ \\[6pt]
%	& + &\E[V_{32}(\mb b_3-\mb b_1)(\mb b_2-\mb b_3)]+ \\[6pt]
%	& + &\E[V_{34}(\mb b_3-\mb b_1)(\mb b_4-\mb b_3)] \\[6pt]
%	& = & - V_{31}K_{44} > 0
%	%& \simeq & 0.0432 / 80 - 0.0953 / 80 \le 0
%\end{array}
%$$
$$
\!\!\! 
\begin{array}{rcl}
	& & \E\left[\left(\mb b_1 - \mb b_3\right)\left(W_1 + \sum_{\gamma}V_{\gamma}(\mb b_\gamma - \mb b_1)\right)\right] \\[6pt]
	& = & \E[V_{12}(\mb b_1-\mb b_3)(\mb b_2-\mb b_1)] \ + \\[6pt]
	& + &\E[V_{13}(\mb b_1-\mb b_3)(\mb b_3-\mb b_1)] \ + \\[6pt]
	& + &\E[V_{14}(\mb b_3-\mb b_1)(\mb b_4-\mb b_1)] \\[6pt]
	& = & - V_{13}K_{44} > 0
	%& \simeq & 0.0432 / 80 - 0.0953 / 80 \le 0
\end{array}
$$
%$$
%\!\!\! 
%\begin{array}{rcl}
%& & \E\left[\left(\mb b_3 - \mb b_1\right)\left(W_3 + \sum_{p}V_{3p}(\mb b_p - \mb b_3)\right)\right] \\[6pt]
%& = &\E[V_{31}(\mb b_3-\mb b_1)(\mb b_1-\mb b_3)] \ + \\[6pt]
%& + &\E[V_{32}(\mb b_3-\mb b_1)(\mb b_2-\mb b_3)] \ + \\[6pt]
%& + &\E[V_{34}(\mb b_3-\mb b_1)(\mb b_4-\mb b_3)] \\[6pt]
%& = & -V_{31} K_{55} > 0\,,
%%& \simeq & 0.0432 / 80 - 0.0953 / 80 \le 0
%\end{array}
%$$
where the equalities follow from \eqref{eq:beta_i_ex}, \eqref{eq:path_ex} and \eqref{eq:K_example}, and the inequality from $V_{13}<0$. This means that, under rule $\pi = z^*$, users that receive message $\gamma_1$ prefer path~$\gamma_3$ over path~$\gamma_1$.
The interpretation for this fact is the following. Let $\E_\gamma[\cdot]$ be the expected value of a random variable according to the posterior belief of users that received message $\gamma$. Due to \eqref{eq:K_example}, $\E_1[\mb b_e] = \E[\mb b_e]$ for every $e$ in $\mathcal E \setminus \{e_4\}$. %Let $b=A'\theta$. 
Since $V_{13}<0$ and since link $e_4$ belongs to path $\gamma_3$ only, $V_{13}<0$ and Remark \ref{rem:Metzler} imply that $z_1^*(\Theta)$ is decreasing in $\mb \mb{\Theta}_4$. This, together with Proposition \ref{prp:bayes}(ii), implies that a user that receives message $\gamma_1$ updates her belief so that $\E_1[\Theta_4] \le \E[\Theta_4]$ and $$\E_1[\mb b_3] \le \E[\mb b_3] = \E[\mb b_\gamma] = \E_1[\mb b_\gamma]\,, \ \ \ \forall \gamma \in \Gamma \setminus \{\gamma_3\}.$$ 
For this reason, users that receive message $\gamma_1$ under rule $\pi = z^*$ have incentive in deviating to path $\gamma_3$. Therefore the rule $\pi = z^*$ is not obedient, which implies by Theorem \ref{rv} that no signaling rule can achieve the optimal cost.
\end{ex}

This examples has shown that, when the network has not parallel links, users that receive a recommendation may want to deviate to an alternative route, even if the free-flow delay of all paths is identical and the network state entries are uncorrelated. In particular, this can happen if worsening a path leads to a reduction of the system-optimum path flow over an alternative path, a fact that is related to the sign of the non-diagonal elements of matrix $V$.

In the final part of the section we discuss the challenges that relaxing Assumption \ref{ass:A} introduces to our problem. 
\begin{ex}
Consider the network of Figure \ref{fig:not_LI}. For this network, there exists a continuum of $z^*$ such that $f^* = Az^*$, due to the fact that $A$ is not injective. The system-optimum flow can be obtained as BWE for a deterministic fair signal policy with rule $\pi$ if at least a signaling rule $\pi = z^*$ is obedient, but understanding which of these rules should be tested is in general hard problem. Moreover, the planner might also add the garbling step, i.e., randomizing the fraction of agents that receive each message $\pi=z^*$, (cf. \cite{massicot2022competitive,koessler2024correlated}) to obtain network flows that cannot be obtained by signaling rules that are deterministic functions of the network state realization, as considered in this paper. Therefore, for networks that do not satisfy Assumption \ref{ass:A}, it may be needed to generalize the notion of signaling rule to allow for such randomization, which is left for future research.
\end{ex}

\begin{figure}
	\centering
	\begin{tikzpicture}
		\node [circle,draw] (1) at (0,0) {$o$};
		\node [circle,draw] (2) at (2,0)
		{$n$}; 
		\node [circle,draw] (3) at (4,0) {$d$};
		
		\path[->] (1) edge [bend left] node {} (2);
		\path[->] (1) edge [bend right] node {} (2);
		\path[->] (2) edge [bend left] node {} (3);
		\path[->] (2) edge [bend right] node {} (3);
	\end{tikzpicture}
	\caption{A network that does not satisfy Assumption \ref{ass:A}. \label{fig:not_LI}}
\end{figure}
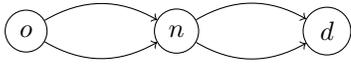

\section{Conclusion}\label{sec:conclusion}
In this paper we considered an information design problem where an omniscient planner observes the realization of a random network state and share private signals to the users to correlate their behavior with the network state and minimize the cost of the resulting equilibrium flow. In particular, we study when the optimal cost is achievable by a fair private signal policy. We show that it is without loss of generality to restrict the attention to direct obedient signaling rules, namely, the set of messages coincides with path recommendations such that the users have no incentive in deviating to alternative paths. We then provide sufficient and necessary conditions for optimality for networks with injective link-path adjacency matrix. Since these conditions are hard to interpret, we then provide more restrictive conditions for optimality that depend on the network topology, showing that stronger conditions are required as the network topology becomes more general. Future research lines aim to extend the analysis for network with non-injective adjacency matrices, by allowing the planner to add an additional level of randomness (i.e., adding the garbling step, in the language of \cite{massicot2025strategic}), as well as considering more general setting whereby multiple information providers compete for customers.

\section*{References}

\bibliographystyle{IEEEtran}
%\nocite{*}
\bibliography{references.bib}

\begin{thebibliography}{10}
\providecommand{\url}[1]{#1}
\csname url@samestyle\endcsname
\providecommand{\newblock}{\relax}
\providecommand{\bibinfo}[2]{#2}
\providecommand{\BIBentrySTDinterwordspacing}{\spaceskip=0pt\relax}
\providecommand{\BIBentryALTinterwordstretchfactor}{4}
\providecommand{\BIBentryALTinterwordspacing}{\spaceskip=\fontdimen2\font plus
\BIBentryALTinterwordstretchfactor\fontdimen3\font minus
  \fontdimen4\font\relax}
\providecommand{\BIBforeignlanguage}[2]{{%
\expandafter\ifx\csname l@#1\endcsname\relax
\typeout{** WARNING: IEEEtran.bst: No hyphenation pattern has been}%
\typeout{** loaded for the language `#1'. Using the pattern for}%
\typeout{** the default language instead.}%
\else
\language=\csname l@#1\endcsname
\fi
#2}}
\providecommand{\BIBdecl}{\relax}
\BIBdecl

\bibitem{cianfanelli2023information}
L.~Cianfanelli, A.~Ambrogio, and G.~Como, ``Information design in bayesian
  routing games,'' in \emph{2023 62nd IEEE Conference on Decision and Control
  (CDC)}.\hskip 1em plus 0.5em minus 0.4em\relax IEEE, 2023, pp. 3945--3949.

\bibitem{ambrogio2024information}
A.~Ambrogio, L.~Cianfanelli, and G.~Como, ``Information design for congestion
  minimization in transportation networks,'' \emph{IFAC-PapersOnLine}, vol.~58,
  no.~30, pp. 181--185, 2024.

\bibitem{wardrop1952road}
J.~G. Wardrop, ``Road paper. some theoretical aspects of road traffic
  research.'' \emph{Proceedings of the institution of civil engineers}, vol.~1,
  no.~3, pp. 325--362, 1952.

\bibitem{beckmann1956studies}
M.~Beckmann, C.~B. McGuire, and C.~B. Winsten, ``Studies in the economics of
  transportation,'' Tech. Rep., 1956.

\bibitem{sandholm2010population}
W.~H. Sandholm, \emph{Population games and evolutionary dynamics}.\hskip 1em
  plus 0.5em minus 0.4em\relax MIT press, 2010.

\bibitem{roughgarden2002price}
T.~Roughgarden, ``The price of anarchy is independent of the network
  topology,'' in \emph{Proceedings of the thiry-fourth annual ACM symposium on
  Theory of computing}, 2002, pp. 428--437.

\bibitem{cianfanelli2023optimal}
L.~Cianfanelli, G.~Como, A.~E. Ozdaglar, and F.~Parise, ``Optimal intervention
  in transportation networks,'' \emph{IEEE Transactions on Automatic Control},
  vol.~68, no.~12, pp. 7073--7088, 2023.

\bibitem{farahani2013review}
R.~Z. Farahani, E.~Miandoabchi, W.~Y. Szeto, and H.~Rashidi, ``A review of
  urban transportation network design problems,'' \emph{European journal of
  operational research}, vol. 229, no.~2, pp. 281--302, 2013.

\bibitem{roughgarden2006severity}
T.~Roughgarden, ``On the severity of {B}raess's paradox: Designing networks for
  selfish users is hard,'' \emph{Journal of Computer and System Sciences},
  vol.~72, no.~5, pp. 922--953, 2006.

\bibitem{marcotte1986network}
P.~Marcotte, ``Network design problem with congestion effects: A case of
  bilevel programming,'' \emph{Mathematical programming}, vol.~34, no.~2, pp.
  142--162, 1986.

\bibitem{cole2006much}
R.~Cole, Y.~Dodis, and T.~Roughgarden, ``How much can taxes help selfish
  routing?'' \emph{Journal of Computer and System Sciences}, vol.~72, no.~3,
  pp. 444--467, 2006.

\bibitem{Como.Maggistro:22}
G.~Como and R.~Maggistro, ``Distributed dynamic pricing of multiscale
  transportation networks,'' \emph{IEEE Transactions on Automatic Control},
  vol.~67, no.~4, pp. 1625--1638, 2022.

\bibitem{fleischer2004tolls}
L.~Fleischer, K.~Jain, and M.~Mahdian, ``Tolls for heterogeneous selfish users
  in multicommodity networks and generalized congestion games,'' in \emph{45th
  Annual IEEE Symposium on Foundations of Computer Science}.\hskip 1em plus
  0.5em minus 0.4em\relax IEEE, 2004, pp. 277--285.

\bibitem{srinivasan2000modeling}
K.~K. Srinivasan and H.~S. Mahmassani, ``Modeling inertia and compliance
  mechanisms in route choice behavior under real-time information,''
  \emph{Transportation Research Record}, vol. 1725, no.~1, pp. 45--53, 2000.

\bibitem{acemoglu2018informational}
D.~Acemoglu, A.~Makhdoumi, A.~Malekian, and A.~Ozdaglar, ``Informational
  {B}raess' paradox: The effect of information on traffic congestion,''
  \emph{Operations Research}, vol.~66, no.~4, pp. 893--917, 2018.

\bibitem{wu2021value}
M.~Wu, S.~Amin, and A.~E. Ozdaglar, ``Value of information in {B}ayesian
  routing games,'' \emph{Operations Research}, vol.~69, no.~1, pp. 148--163,
  2021.

\bibitem{das2017reducing}
S.~Das, E.~Kamenica, and R.~Mirka, ``Reducing congestion through information
  design,'' in \emph{2017 55th annual Allerton conference on communication,
  control, and computing (Allerton)}.\hskip 1em plus 0.5em minus 0.4em\relax
  IEEE, 2017, pp. 1279--1284.

\bibitem{tavafoghi2017informational}
H.~Tavafoghi and D.~Teneketzis, ``Informational incentives for congestion
  games,'' in \emph{2017 55th Annual Allerton Conference on Communication,
  Control, and Computing (Allerton)}.\hskip 1em plus 0.5em minus 0.4em\relax
  IEEE, 2017, pp. 1285--1292.

\bibitem{massicot2022competitive}
O.~Massicot and C.~Langbort, ``Competitive comparisons of strategic information
  provision policies in network routing games,'' \emph{IEEE Transactions on
  Control of Network Systems}, vol.~9, no.~4, pp. 1589--1599, 2022.

\bibitem{zhu2022information}
Y.~Zhu and K.~Savla, ``Information design in nonatomic routing games with
  partial participation: Computation and properties,'' \emph{IEEE Transactions
  on Control of Network Systems}, vol.~9, no.~2, pp. 613--624, 2022.

\bibitem{wu2019information}
M.~Wu and S.~Amin, ``Information design for regulating traffic flows under
  uncertain network state,'' in \emph{2019 57th Annual Allerton Conference on
  Communication, Control, and Computing (Allerton)}.\hskip 1em plus 0.5em minus
  0.4em\relax IEEE, 2019, pp. 671--678.

\bibitem{dughmi2016algorithmic}
S.~Dughmi and H.~Xu, ``Algorithmic {B}ayesian persuasion,'' in
  \emph{Proceedings of the forty-eighth annual ACM symposium on Theory of
  Computing}, 2016, pp. 412--425.

\bibitem{meigs2020optimal}
E.~Meigs, F.~Parise, A.~Ozdaglar, and D.~Acemoglu, ``Optimal dynamic
  information provision in traffic routing,'' \emph{arXiv preprint
  arXiv:2001.03232}, 2020.

\bibitem{tavafoghi2019strategic}
H.~Tavafoghi and D.~Teneketzis, ``Strategic information provision in routing
  games,'' 2019.

\bibitem{ouyang2024anapproach}
Y.~Ouyang, H.~Tavafoghi, and D.~Teneketzis, ``An approach to stochastic dynamic
  games with asymmetric information and hidden actions,'' \emph{Dynamic Games
  and Applications}, pp. 1--34, 2024.

\bibitem{doval2024constrained}
L.~Doval and V.~Skreta, ``Constrained information design,'' \emph{Mathematics
  of Operations Research}, vol.~49, no.~1, pp. 78--106, 2024.

\bibitem{gould2023rationality}
B.~T. Gould and P.~N. Brown, ``Rationality and behavior feedback in a model of
  vehicle-to-vehicle communication,'' in \emph{2023 62nd IEEE Conference on
  Decision and Control (CDC)}.\hskip 1em plus 0.5em minus 0.4em\relax IEEE,
  2023, pp. 3232--3237.

\bibitem{feng2024rationality}
Y.~Feng, C.-J. Ho, and W.~Tang, ``Rationality-robust information design:
  Bayesian persuasion under quantal response,'' in \emph{Proceedings of the
  2024 Annual ACM-SIAM Symposium on Discrete Algorithms (SODA)}.\hskip 1em plus
  0.5em minus 0.4em\relax SIAM, 2024, pp. 501--546.

\bibitem{sezer2023robust}
F.~Sezer and C.~Eksin, ``Robust social welfare maximization via information
  design in linear-quadratic-{G}aussian games,'' \emph{IEEE Control Systems
  Letters}, pp. 3096 -- 3101, 2023.

\bibitem{liu2016effects}
J.~Liu, S.~Amin, and G.~Schwartz, \emph{Effects of Information Heterogeneity in
  Bayesian Congestion games}.\hskip 1em plus 0.5em minus 0.4em\relax
  Transportation Science, manuscript no. 1, 2016.

\bibitem{massicot2019public}
O.~Massicot and C.~Langbort, ``Public signals and persuasion for road network
  congestion games under vagaries,'' \emph{IFAC-PapersOnLine}, vol.~51, no.~34,
  pp. 124--130, 2019.

\bibitem{massicot2024OnNetwork}
------, ``On network congestion reduction using public signals under boundedly
  rational user equilibria (full version),'' 06 2024.

\bibitem{griesbach2024information}
S.~M. Griesbach, M.~Hoefer, M.~Klimm, and T.~Koglin, ``Information design for
  congestion games with unknown demand,'' in \emph{Proceedings of the AAAI
  Conference on Artificial Intelligence}, vol.~38, no.~9, 2024, pp. 9722--9730.

\bibitem{ferguson2024information}
B.~L. Ferguson, P.~N. Brown, and J.~R. Marden, ``Information signaling with
  concurrent monetary incentives in {B}ayesian congestion games,'' \emph{IEEE
  Transactions on Intelligent Transportation Systems}, 2024.

\bibitem{koessler2024correlated}
F.~Koessler, M.~Scarsini, and T.~Tomala, ``Correlated equilibria in large
  anonymous {B}ayesian games,'' \emph{Mathematics of Operations Research},
  2024.

\bibitem{massicot2025strategic}
O.~Massicot, ``On the strategic transmission of information to imperfect agents
  and crowds,'' Ph.D. dissertation, 2025.

\bibitem{bergemann2019information}
D.~Bergemann and S.~Morris, ``Information design: A unified perspective,''
  \emph{Journal of Economic Literature}, vol.~57, no.~1, pp. 44--95, 2019.

\bibitem{myerson1982optimal}
R.~B. Myerson, ``Optimal coordination mechanisms in generalized
  principal--agent problems,'' \emph{Journal of mathematical economics},
  vol.~10, no.~1, pp. 67--81, 1982.

\bibitem{koessler2025full}
F.~Koessler, M.~Scarsini, and T.~Tomala, ``Full implementation via information
  design in nonatomic games,'' \emph{arXiv preprint arXiv:2502.05920}, 2025.

\bibitem{koessler2022information}
------, ``Information design in large anonymous games,'' working paper, Tech.
  Rep., 2022.

\bibitem{bergemann2016bayes}
D.~Bergemann and S.~Morris, ``Bayes correlated equilibrium and the comparison
  of information structures in games,'' \emph{Theoretical Economics}, vol.~11,
  no.~2, pp. 487--522, 2016.

\bibitem{milchtaich2006network}
I.~Milchtaich, ``Network topology and the efficiency of equilibrium,''
  \emph{Games and Economic Behavior}, vol.~57, no.~2, pp. 321--346, 2006.

\end{thebibliography}

\appendices
\section{Proofs}

\subsection{Proof of Proposition \ref{prp:pot}}\label{app:A}
Observe that $\Phi_\pi(y)$ is convex because the delay functions are increasing in the flow and $f^{\pi,y}$ is linear in $y$ (cf. \eqref{eq:flow_y}), hence $\Phi_\pi(y)$ is the composition of convex functions. Therefore, $y$ is a solution of \eqref{eq:y} if and only it satisfies the KKT conditions. Since $y$ must satisfy constraints \eqref{eq:con:y}, the KKT conditions are the following: for every two nodes $o,d$ in $\mc V$, $i$ in $\Gamma_{od}$, $m$ in $\mathcal M$ such that $y^{od}_{im}>0$, it must hold true
\be\label{eq:optimal_y}
\frac{\partial \Phi_\pi(y)}{\partial y_{jm}^{od}} -  \frac{\partial \Phi_\pi(y)}{\partial y_{im}^{od}} \ge 0\,, \quad \forall j \in \Gamma_{od}\,. 
\ee
Plugging \eqref{eq:pop_game} into \eqref{eq:optimal_y} and using \eqref{eq:cost_path} we conclude that a necessary and sufficient conditions for optimality is that
$$
\int_{\R_+^{\mc E}} (c_i(\theta, f^{\pi,y}(\theta))-c_j(\theta,f^{\pi,y}(\theta)) \pi_m^{od} (\theta) d\P(\theta) \leq 0\,,
$$
for every $j$ in $\Gamma_{od}$,
which is the definition of BWE (cf. Definition \ref{def:bue}). 
The uniqueness of the BWE follows from the strict monotonicity of the delay functions.

%%%%%%%%%%%%%%%%%%%%%%%%%%%%%%%%%%%%%%%%%%%%%%%%%%%%%%%%%%%%%%%%%%
\subsection{Proof of Proposition \ref{prp:no_full}}\label{app:linear}
	Consider the full support $y^{FI}$ such that 
	\be\label{eq:full_y}
	(y^{FI})_{i\theta}^{od}>0\, \quad \forall o,d \in \mc V\,, \forall i \in \Gamma_{od}\,, \forall \theta \in \text{supp}(\Theta)
	\ee
	and the associated full information BWE as per \eqref{eq:flow_y}, denoted $f^{\pi^{FI}}$ in short.
	The equilibrium condition \eqref{eq:full_info} and \eqref{eq:full_y} together imply that $$c_i(\theta,f^{\pi^{FI}}(\theta)) = c_j(\theta,f^{\pi^{FI}}(\theta))$$ 
	for every two nodes $o,d$ in $\mc V$, $i,j$ in $\Gamma_{od}$ and $\theta$ in $\text{supp}(\Theta)$. This is equivalent to assume the existence of a function $\lambda: \R_{+}^{\mc E} \to \R_+^{\mc V \times \mc V}$ such that
$$
c_\gamma(\theta,f^{\pi^{FI}}(\theta)) = %(A'[\alpha]^{-1}Az(\theta) + A'\theta)_i = 
\lambda_{od}(\theta)\,, \quad \forall \theta \in \text{supp}(\Theta), \gamma \in \Gamma_{od}\,.$$
Note that $\lambda(\theta)$ is linear because of Assumption~\ref{ass:linear}.
Hence, the cost of the full information BWE is
$$
\begin{aligned}
	& C(f^{\pi^{FI}}) \\[4pt]
	= & \ds \int_{\R_{+}^{\mc E}}\sum_{e \in \mc E} f_e^{\pi^{FI}}(\theta) \tau_e(\theta_e,f_e^{\pi^{FI}}(\theta)) d\P(\theta) \\[4pt]
	= &\ds \int_{\R_{+}^{\mc E}}\sum_{e \in \mc E} \sum_{o,d \in \mc V} \sum_{\gamma \in \Gamma_{od}} (y^{FI})^{od}_{\gamma\theta} \ups_{od} A_{e\gamma} \tau_e(\theta_e,f_e^{\pi^{FI}}(\theta)) d\P(\theta) \\[4pt]
	= & \ds \int_{\R_{+}^{\mc E}} \sum_{o,d \in \mc V} \sum_{\gamma \in \Gamma_{od}} c_i^{od}(\theta,f^{\pi^{FI}}(\theta)) (y^{FI})^{od}_{\gamma\theta} \ups_{od} d\P(\theta) \\[4pt]
	= & \ds \int_{\R_{+}^{\mc E}} \sum_{o,d \in \mc V} \lambda_{od}(\theta) \sum_{\gamma \in \Gamma_{od}} (y^{FI})^{od}_{\gamma\theta} \ups_{od} d\P(\theta) \\[6pt]
	= & \ds \int_{\R_{+}^{\mc E}} \sum_{o,d \in \mc V} \lambda_{od}(\theta) \ups_{od} d\P(\theta) \\[4pt]
	= & \sum_{o,d \in \mc V} \ups_{od} \lambda_{od}(\E[\Theta])
	%	& = &\ds \int_{\mathbf\Theta}\sum_\gamma \sum_{i \in \mc P_\gamma} (A'[\alpha]^{-1}Az(\theta) + A'\theta)_i z_i(\theta) d\P(\theta) \\[6pt]
	%	& = &  \ds \int_{\mathbf\Theta}\sum_\gamma \lambda_\gamma(\theta) \sum_{i \in \mc P_\gamma} z_i(\theta) d\P(\theta) \\[6pt]
	%	& = &\ds \int_{\mathbf\Theta}\sum_\gamma \lambda_\gamma(\theta) v_\gamma d\P(\theta) \\[6pt]
	%	& = &\ds \sum_\gamma \lambda_\gamma(\E[\Theta]) v_\gamma \\[6pt]
\end{aligned}
$$
where the second equality follows from \eqref{eq:full_info_flow}, the third one from \eqref{eq:cost_path}, the penultimate one from \eqref{eq:con:y} and the last one from the linearity of $\lambda_\gamma(\theta)$.
Consider now the no information rule $\pi^{NI}$. Note from \eqref{eq:no_info_eq} that under Assumption \ref{ass:linear} the no information BWE $f^{\pi^{NI}}(\theta)$ depends on $\E[\Theta]$ but does not depend on the network state realization $\theta$. 
%For this reason, with a slight abuse of notation, from now one we shall denote such equilibrium by $f^{\pi^{NI}}(\E[\Theta])$. 
For this reason, from now one we shall denote such equilibrium by $f^{\pi^{NI}}$ omitting the dependence on $\theta$.
Moreover, it follows from \eqref{eq:no_info_eq} and \eqref{eq:full_info} that
\be\label{eq:full_no}
f^{\pi^{FI}}(\E[\Theta]) = f^{\pi^{NI}}\,.
%(\E[\Theta])\,.
\ee
Since there exists $y^{NI}$ that satisfies \eqref{eq:y} for the rule $\pi^{NI}$ and such that $(y^{NI})^{od}_{\gamma m^*}>0$ for every two nodes $o,d$ in $\mc V$ and $\gamma$ in $\Gamma_{od}$, the equilibrium condition \eqref{eq:no_info_eq} is equivalent to $$c_i(\E[\Theta],f^{\pi^{NI}}) = c_j(\E[\Theta],f^{\pi^{NI}})$$ 
for every two nodes $o,d$ in $\mc V$ and two paths $i,j$ in $\Gamma_{od}$. This in turn is equivalent by \eqref{eq:full_no} to
\be\label{eq:lambda_ni}
\!\!\! c_\gamma(\E[\Theta],f^{\pi^{NI}}) = c_\gamma(\E[\Theta],f^{\pi^{FI}}(\E[\Theta])) =
\lambda_{od}(\E[\Theta])\,,
\ee
for every two nodes $o,d$ in $\mc V$ and $\gamma$ in $\Gamma_{od}$. Hence,
$$
\begin{aligned}
	& C(f^{\pi^{NI}}) \\[4pt] 
	= & \ds \int_{\R_+^{\mc E}}\sum_{e \in \mc E} f_e^{\pi^{NI}} \tau_e(\theta_e,f_e^{\pi^{NI}}) d\P(\theta) \\[4pt]
	= & \ds \int_{\R_+^{\mc E}}\sum_{e \in \mc E} \sum_{o,d \in \mc V} \sum_{\gamma \in \Gamma_{od}} (y^{NI})^{od}_{\gamma m^*} \ups_{od} A_{e\gamma} \tau_e(\theta,f_e^{\pi^{NI}}) d\P(\theta) \\[4pt]
	= & \ds \int_{\R_+^{\mc E}} \sum_{o,d \in \mc V} \sum_{\gamma \in \Gamma_{od}} c_\gamma(\theta,f^{\pi^{NI}}) (y^{NI})^{od}_{\gamma m^*} \ups_{od} d\P(\theta) \\[4pt]
	= & \sum_{o,d \in \mc V} \sum_{\gamma \in \Gamma_{od}} c_\gamma(\E[\Theta],f^{\pi^{NI}}) (y^{NI})^{od}_{\gamma m^*} \ups_{od} \\[4pt]
	= & \sum_{o,d \in \mc V} \lambda_{od}(\E[\Theta]) \sum_{\gamma \in \Gamma_{od}} (y^{NI})^{od}_{\gamma m^*} \ups_{od} \\[4pt]
	= & \sum_{o,d \in \mc V} \ups_{od} \lambda_{od}(\E[\Theta]) = C(f^{\pi^{FI}})\,.
\end{aligned}
$$
where the second equality follows from \eqref{eq:f_no_info}, the third one from \eqref{eq:cost_path}, the fourth one from the linearity of the delay functions, the penultimate one from \eqref{eq:lambda_ni} and the last one from \eqref{eq:con:y}.
This concludes the proof.

\subsection{Proof of Proposition \ref{prp:suff_general}}\label{app:metzler}
Note that $\E[(\mb b_i - \mb b_j)W_i] = W_i\E[\Delta_{ij}] = 0$ because $\Delta = \0$. Hence, by Theorem \ref{thm:opt_gen}, $\pi = z^*$ is obedient if and only if
\be\label{eq:ob_arb}
\E\Big[\sum_{\gamma \in \Gamma}V_{i\gamma}(\mb b_\gamma - \mb b_i)(\mb b_i - \mb b_j)\Big] \le 0\,.
\ee
For a link $e$ in $\mc E$ and a path $p$ in $\Gamma$, we say that $e \in p$ if $A_{ep} = 1$. Hence,
\begin{equation*}\label{eq:b_path}
\ba{rcl}
\E[\mb b_p\mb b_q] & = & \ds\E\Big[\Big(\sum_{m \in p}\Theta_m\Big)\Big(\sum_{n \in q}\Theta_n\Big)\Big]\\[10pt]
& = & \ds\sum_{m \in p}\sum_{n \in q}\E[\Theta_m \Theta_n]\\[10pt]
& = & \ds \sum_{m \in p}\sum_{n \in q} \E[\Theta_m]\E[\Theta_n] + \sum_{\substack{e \in \mathcal E: \\ e \in p \cap q}} K_{ee}\\[10pt]
& = & \ds \sum_{m \in p}\E[\Theta_m]\sum_{n \in q}\E[\Theta_n] + \sum_{\substack{e \in \mathcal E: \\ e \in p \cap q}} K_{ee}\\[10pt]
& = &\ds \E[\mb b_p]\E[\mb b_q] + \sum_{\substack{e \in \mathcal E: \\ e \in p \cap q}} K_{ee}\,,
\ea
\end{equation*}
where the third equivalence follows from the fact that $K$ is a diagonal matrix.
Using this expression and $\Delta = \0$,
\be\label{eq:diff_b}
\!\!\!\!\!\!\!\ba{rcl}
& & \E[(\mb b_p-\mb b_i)(\mb b_i-\mb b_j)] \\[8pt] 
& = & \ds \sum_{\substack{e \in \mathcal E: \\ e \in i \cap p}} K_{ee} - \sum_{\substack{e \in \mathcal E: \\ e \in i}} K_{ee} - \sum_{\substack{e \in \mathcal E: \\ e \in j \cap p}} K_{ee} + \sum_{\substack{e \in \mathcal E: \\ e \in i \cap j}} K_{ee},
\ea
\ee
Moreover,
$$
\ba{rcl}
\ds \sum_{\substack{e \in \mathcal E: \\ e \in i}} K_{ee}- \sum_{\substack{e \in \mathcal E: \\ e \in i \cap p}} K_{ee} & = & \ds \sum_{\substack{e \in \mathcal E: \\ e \in i \\ e \notin p}} K_{ee}\,, \\[30pt]
\ds \sum_{\substack{e \in \mathcal E: \\ e \in i \cap j}} K_{ee} - \sum_{\substack{e \in \mathcal E: \\ e \in j \cap p}} K_{ee} & = & \ds \sum_{\substack{e \in \mathcal E: \\ e \in i \cap j \\ e \notin p}} K_{ee} - \sum_{\substack{e \in \mathcal E: \\ e \in j \cap p \\ e \notin i}} K_{ee}\,.
\ea$$
Plugging these two expressions into \eqref{eq:diff_b} we get
\begin{equation*}\label{eq:quad_beta}
\begin{aligned}
\E[(\mb b_p-\mb b_i)(\mb b_i-\mb b_j)] &= -\sum_{\substack{e \in \mathcal E: \\ e \in i \\ e \notin p}} \! K_{ee} + \sum_{\substack{e \in \mathcal E: \\ e \in i \cap j \\ e \notin p}} \!\! K_{ee} - \sum_{\substack{e \in \mathcal E: \\ e \in j \cap p \\ e \notin i}} \!\! K_{ee} \\
&= - \sum_{\substack{e \in \mathcal E: \\ e \in i \\ e \notin j \cup p}} K_{ee} - \sum_{\substack{e \in \mathcal E: \\ e \in j \cap p \\ e \notin i}} K_{ee} \\
&\le 0,
\end{aligned}
\end{equation*}
where the last inequality follows from the fact that the diagonal of $K$ is non-negative by construction. The statement is proved by plugging this equation into \eqref{eq:ob_arb} and using the fact $V$ is Metzler.

\begin{IEEEbiography}[{\includegraphics[width=1in,height=1.25in,clip,keepaspectratio]{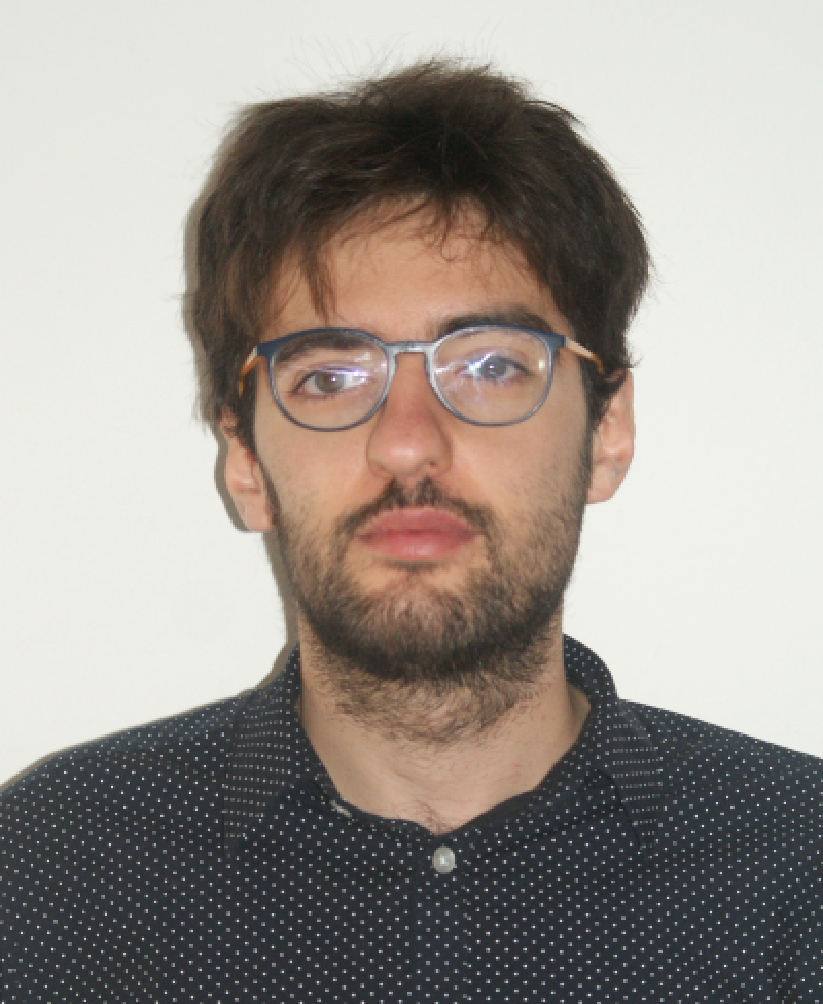}}]{Leonardo Cianfanelli} (M'21) received the B.Sc. (cum Laude) in physics and astrophysics in 2014 from Università degli Studi di Firenze, Italy, the M.S. in physics of complex systems (cum Laude) in 2017 from Università di Torino, Italy, and the PhD in pure and applied mathematics in 2022 from Politecnico di Torino. He is currently an assistant professor at the Department  of  Mathematical Sciences, Politecnico di Torino, Italy. He was a Visiting Student at Laboratory for Information and Decision Systems, Massachusetts Institute of Technology, Cambridge, MA, USA, in 2018-2020 and then research assistant at Politecnico di Torino, Italy, in 2021-2025. His research focuses on dynamics and control in network systems, with applications to transportation networks, epidemics.
\end{IEEEbiography}

\begin{IEEEbiography}[{\includegraphics[width=1in,height=1.25in,clip,keepaspectratio]{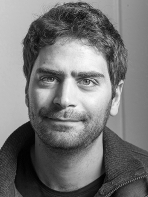}}]
	{Giacomo Como}(M'12) is  a  Professor at  the Department  of  Mathematical  Sciences,  Politecnico di  Torino,  Italy,  and a Senior Lecturer  at  the  Automatic  Control  Department, Lund  University,  Sweden.  He  received the B.Sc., M.S., and Ph.D.~degrees in Applied Mathematics  from  Politecnico  di  Torino,  in  2002,  2004, and 2008, respectively. He was a Visiting Assistant in  Research  at  Yale  University  in  2006--2007  and  a Postdoctoral  Associate  at  the  Laboratory  for  Information  and  Decision  Systems,  Massachusetts  Institute of Technology, from 2008 to 2011. Prof.~Como currently serves as Senior Editor for the \textit{IEEE Transactions on Control of Network Systems}, as Associate  Editor  for \textit{Automatica}, and as the  chair  of the  {IEEE-CSS  Technical  Committee  on  Networks  and  Communications}. He served as Associate Editor for the  \textit{IEEE Transactions on Network Science and Engineering} (2015-2021) and for the \textit{IEEE Transactions on Control of Network Systems} (2016-2022).  He was  the  IPC  chair  of  the  IFAC  Workshop  NecSys'15  and  a  semiplenary speaker  at  the  International  Symposium  MTNS'16.  He  is  recipient  of  the 2015  George S. ~Axelby  Outstanding Paper Award.  His  research interests  are in  dynamics,  information,  and  control  in  network  systems  with  applications to  infrastructure,  social,  and  economic networks.
\end{IEEEbiography}

\begin{IEEEbiography}[{\includegraphics[width=1in,height=1.25in,clip,keepaspectratio]{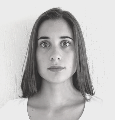}}]{Alexia Ambrogio} (S'25) received the B.Sc. in Mathematics for Engineering in 2020 from Politecnico di Torino, Italy, and the M.S. in Mathematical Engineering in 2023 from Politecnico di Torino, Italy. She is currently a PhD student in Mathematical Sciences at Politecnico di Torino, Italy. She was a Visiting Student at Gipsa-Lab, University of Grenoble Alpes, France, in 2025. Her research focuses on network dynamics and game theory, applied to transportation networks.
\end{IEEEbiography}

\end{document}